\documentclass{scrartcl}
\usepackage[margin=2.5cm,bottom=1cm,includefoot]{geometry}
\usepackage{amsmath,amssymb,amsthm}
\usepackage{bbm}
\usepackage{array}
\usepackage{mathrsfs}
\usepackage{enumerate}
\usepackage{graphicx}
\usepackage[utf8]{inputenc}
\usepackage{hyperref}
\usepackage[all]{hypcap}
\usepackage{float}

\usepackage{color}

\newtheorem{thm}{Theorem}
\newtheorem{lem}[thm]{Lemma}

\theoremstyle{remark}

\let\BFseries\bfseries\def\bfseries{\BFseries\mathversion{bold}}

\newcommand{\N}{\mathbb{N}}
\newcommand{\R}{\mathbb{R}}
\newcommand{\C}{\mathbb{C}}
\newcommand{\E}{\mathbb{E}}

\renewcommand{\P}{\mathbb{P}}
\newcommand{\Q}{\mathbb{Q}}
\newcommand{\CC}{\mathcal{C}}
\newcommand{\NN}{\mathcal{N}}
\newcommand{\BB}{\mathcal{B}}
\newcommand{\FF}{\mathcal{F}}
\newcommand{\1}{\mathbbm{1}}

\newcommand{\epsi}{\varepsilon}

\newcommand{\abs}{|}

\begin{document}

\title{Ornstein--Uhlenbeck process conditioned to have restricted $L_2$-norm}
\author{Frank Aurzada\qquad Mikhail Lifshits\qquad Max Wiegand}
\date{August 21, 2026}

\maketitle

\begin{abstract}
    We condition an Ornstein--Uhlenbeck process on having an atypically small $L_2$-norm on long time intervals.
    The weak limit of these conditioned processes is again an Ornstein--Uhlenbeck process, this time with a stronger mean-reverting force than the unconditioned process, which is controlled by the restriction on the $L_2$-norm. 
\end{abstract}



\allowdisplaybreaks

\section{Introduction and Main Results}
There is a long history of conditioning Brownian motion on highly improbable events and studying the weak limits of such conditioned processes. This starts with the pioneering work of Doob  \cite{doob57}  on Brownian motion conditioned to be positive. Similarly classic is the work of Knight \cite{KnightBM} on Brownian motion not leaving a fixed interval, which can be interpreted as conditioning on having small $L_\infty$-norm. Motivated by this, Aurzada, Lifshits, and Schickentanz, \cite{ALSbmtoou}, conditioned Brownian motion on having small $L_2$-norm. They show that the resulting limiting process is an Ornstein--Uhlenbeck process.\\
In the present paper, we show the analogous result when we start with the Ornstein--Uhlenbeck process.
Both mentioned results are examples of how restricting paths of stochastic processes to stay closer to the origin than usual is equivalent to an emergent mean-reverting effect controlling the process.
An important difference to \cite{KnightBM} is that a restricted $L_\infty$-norm is a hard boundary for the paths at any point, rather than a locally soft force that the restriction on the $L_2$-norm imposes. That explains why the limiting taboo process in \cite{KnightBM} has a mean-reverting force that becomes singular (i.e.\ repelling) at the boundary.

In order to formulate the main result, let us introduce the necessary notation. Let $\gamma>0$ and let $(U_t)_{t\geq 0}$ be an Ornstein--Uhlenbeck process with parameter $\gamma$ started in $x$, i.e.\ a Gaussian process with covariance kernel $K_U(t,s)=\frac{1}{2\gamma}\left(e^{-\gamma\abs t-s\abs}-e^{-\gamma(t+s)}\right)$ and expectation $m_U(t)=x\,e^{-\gamma t}$. Let $Z=(Z_t)_{t \geq 0}$ be the $L_2$-norm process defined by
\begin{equation*}
    Z_t:=\int_{0}^t U_s^2\,ds,\qquad t\geq 0.
\end{equation*}

Using this notation, we can formulate our main result.

\begin{thm}\label{thm:main}
    Let $0<\theta<\frac{1}{2\gamma}$. Assume there are $K>0$ and $\beta\in(0,1/2)$ such that the sequence $y_T$ satisfies $|y_T-\theta T|\leq K T^\beta$. Then
    \begin{equation*}
        \P_x\left((U_t)_{t\geq0}\in~\cdot~\vert~Z_T \leq y_T\right)
        \overset{w}{\longrightarrow}\mathcal{L}(\tilde{U}), \qquad \text{as $T\to\infty$,}
    \end{equation*}
     where $\tilde{U}$ is an Ornstein--Uhlenbeck process started at $x$ with parameter $\frac{1}{2\theta}$ and $\overset{w}{\longrightarrow}$ stands for weak convergence on $\CC([0,\infty))$. 
\end{thm}

\textbf{Comments.}  The simplest sequence covered by Theorem~\ref{thm:main} is $y_T=\theta T$. 
It is straightforward to compute that $\E[Z_T]\sim \frac{1}{2\gamma}T$. This means that the typical size of $Z_T$ is of order $\frac{1}{2\gamma} T$; and restricting $Z_T$ to be smaller than $y_T\sim \theta T$ with $\theta<\frac{1}{2\gamma}$ means that $Z_T\leq y_T$ is a small ball event. We will control the corresponding small ball probability for $y_T=\theta T$ in Theorem~\ref{thm:sbp} below and for more general sequences $y_T$ and starting points $x$ in a uniform version in Theorem~\ref{thm:uniformasymptotics} below.

\medskip
Let us recall that \cite{ALSbmtoou} showed that, for a Brownian motion $(B_t)_{t\geq 0}$ and $\theta_1>0$,
$$
\P_x\left((B_t)_{t\geq0}\in~\cdot~\vert~\int_0^T B_s^2 d s \leq \theta_1 T\right)
        \overset{w}{\longrightarrow}\mathcal{L}(U),
$$
where $U$ is an Ornstein--Uhlenbeck process with parameter $\frac{1}{2\theta_1}$. Now, our result shows that conditioning ``a second time'', namely restricting $U$ on $\int_0^T U_s^2 d s \leq \theta_2 T$ with $0<\theta_2<\theta_1$ leads again to an Ornstein--Uhlenbeck process with, this time, parameter $\frac{1}{2\theta_2}$, which is the same process as if we had conditioned Brownian motion on $\int_0^T B_s^2 d s \leq \theta_2 T$ in the first place. One is tempted to think that therefore Theorem~\ref{thm:main} can be obtained from \cite{ALSbmtoou} by passing directly to the stronger conditioning. This argument does not work for two reasons: First, its proof would involve an unjustified exchange of limits. Second, the uniform version w.r.t.\ the sequence $y_T$ cannot be inferred from \cite{ALSbmtoou}. Nonetheless, this double-conditioning is a good heuristic for our result.

\medskip
Our proof uses a result on small ball probabilities. 
We state the result here in a simple asymptotic form and generalize it to uniform bounds in Section~\ref{sec:unifsbp}, which is of interest in its own right.
\begin{thm}\label{thm:sbp}
    Let $0<\theta<\frac{1}{2\gamma}$. Then, as $T\to\infty$,
    \begin{equation*}
        \P_x\left(Z_T\leq \theta T\right)
        \sim C_\theta\,T^{-1/2}\,\exp\left(-\frac{\left(1-2\gamma\theta\right)^2}{8\theta}T-x^2\,\frac{1-2\gamma\theta}{4\theta}\right),
    \end{equation*}
    with
    \[C_{\theta}:=\frac{1}{\sqrt{\pi}}\frac{4\sqrt{\theta}}{\left(1+2\gamma\theta\right)^{3/2}\left(1-2\gamma\theta\right)}.\]
\end{thm}

\textbf{Related work.} Shortly before the completion of this work, the parallel preprint \cite{papertobiasschmidt} appeared. The special case $y_T=\theta T+z$ with fixed $z\in\R$ of our Theorem~\ref{thm:main} as well as our Theorem~\ref{thm:sbp} are contained in their Corollary~3.2. Our Theorem~\ref{thm:uniformasymptotics} below  additionally provides uniformity in the small ball probability result when both the starting point and the constraint are allowed to vary with $T$. The two papers also differ in tools, perspective, and language: ours is formulated in the language of stochastic analysis and change of measure, whereas \cite{papertobiasschmidt} adopts a statistical-mechanics and spectral-theoretic viewpoint based on path ensembles, Feynman–Kac semigroups, and analytic perturbation theory.

Another paper that has to be mentioned is \cite{Fatalov2009}. There, a result similar to
our Theorem~\ref{thm:sbp} can be found. However, it contains some minor errors, and the proof has gaps.

Theorem~\ref{thm:sbp} for $x=0$ is classical, a version of this can be found e.g.\ in \cite{BercuRouault2002}. Let us stress that we need (a uniform version of) Theorem~\ref{thm:sbp} for general $x\in\R$ for the proof of Theorem~\ref{thm:main}. For further literature on $L_2$ small ball estimates see e.g.\ \cite{Li2001} and the survey on the subject \cite{NazarovPetrova2023} and references therein. 
Other interesting and conceptually related papers are \cite{BrycDembo1997,BercuGamboaLavielle2000,ChetriteTouchette2015,duBuissonTouchette2023}.

\medskip

\textbf{Outline of the paper.} This article is structured as follows. As mentioned, we prove uniform bounds for the asymptotic behaviour of the small ball probabilities in Section~\ref{sec:unifsbp}, which proves Theorem~\ref{thm:sbp} as well.
Section~\ref{sec:proofmainthm} then contains the proof of Theorem~\ref{thm:main}, where we use Girsanov's theorem to show that under a change of measure, the conditioned process has the claimed law.

\section{Uniform Small Ball Probabilities}\label{sec:unifsbp}

\subsection{Statement}

As mentioned before Theorem \ref{thm:sbp}, we need a version of the small ball asymptotics for the proof of Theorem \ref{thm:main}
that is uniform in $x$ and $y$.

\begin{thm}\label{thm:uniformasymptotics}
    Let $0<\theta<\frac{1}{2\gamma}$. Then, for any $\delta\in(0,1)$, $\beta\in(0,\frac12)$ and $K>0$, there exists a $T_{\delta,\beta,K}>0$, such that for all $T\geq T_{\delta,\beta,K}$ and any $x\in\R$, $y>0$ with
    \[\abs x\abs\leq T^{\beta/2}\quad\text{and}\quad \left|y-\theta T\right|\leq K\,T^\beta\]
    it holds
    \begin{equation*}
        (1-\delta)\,C_\theta\,A(T)\leq\P_x(Z_T\leq y)\leq (1+\delta)\,C_\theta\,A(T),
    \end{equation*}
    with
    \[C_{\theta}:=\frac{1}{\sqrt{\pi}}\frac{4\sqrt{\theta}}{\left(1+2\gamma\theta\right)^{3/2}\left(1-2\gamma\theta\right)},\]
    and
    \[A(T):=T^{-1/2}\,\exp\left(-\frac{1}{8\frac{y}{T}}\left(1-2\gamma \frac{y}{T}\right)^2T-\frac{x^2}{4\theta}\left(1-2\gamma\theta\right)\right).\]
\end{thm}

\subsection{Proof}

In the first two subsections of the proof we formulate the Bromwich inversion of the Laplace transform of $Z_T$ and decompose the integrand into a more manageable \textit{main} integrand, which contains the asymptotically contributing parts of the integrand, and an \textit{error} integrand.
Then we give a uniform asymptotic bound for the error integral.
We do the same for the main integral, however we consider the contributions of the main integrand along different parts of the Bromwich line separately.
We finally put all those separate asymptotic bounds together, identify the remaining asymptotically contributing terms and evaluate them. For ease of notation we introduce $z:=\frac{y}{T}>0$. Note that $z\to\theta$ for $T\to\infty$.
    
\subsubsection{Bromwich inversion}
    
Note that, in distribution, $U$ can be written as
\begin{equation*}
    U_t=\frac{1}{\sqrt{2\gamma}}\, e^{-\gamma t}W_{e^{2\gamma t}-1}
\end{equation*}
where $W$ is a Brownian motion started in $\sqrt{2\gamma}x$. Formula 7.1.9.3 in \cite{BorodinSalminen} (also see \cite{Dankel1991}) provides the Laplace transform of $Z_T$ (for $\Re u\ge 0$ and $T\geq 0$)
\begin{equation}\label{eq:LT}
    \mathbb{E}_x\!\left[e^{-u Z_T}\right]
    =\frac{\sqrt{g(u)}\,e^{\gamma T/2}}{\sqrt{\sinh(\gamma T g(u))+g(u)\cosh(\gamma T g(u))}}\,
    \exp\!\left(
    -\frac{\gamma x^2\,(g(u)^2-1)\,\sinh(\gamma T g(u))}{2\,\left(\sinh(\gamma T g(u))+g(u)\cosh(\gamma T g(u))\right)}
    \right),
\end{equation}
where
\begin{equation}\label{eq:gdef}
    g(u):=\sqrt{1+\frac{2u}{\gamma^2}}\qquad \text{(principal branch)},
\end{equation}
and we can use the square root as we only consider $\Re u\geq 0$.
Since $M_T(u):=\mathbb{E}_x\left[e^{-uZ_T}\right]$ is analytic on $\Re u>0$, for $y>0$ and any $c>0$, the cumulative distribution function satisfies the inversion formula
\begin{equation}\label{eq:bromwich}
    \P_x(Z_T\leq y)=\lim_{H\to\infty}\frac{1}{2\pi i}\int_{c-iH}^{c+iH}\frac{M_T(u)}{u}\,e^{u y}\,du.
\end{equation}
Note that we switch between the complex line integrals and their path integral representation
\[\frac{1}{2\pi i}\int_{c-iH}^{c+iH}f(u)\,du=\frac{1}{2\pi}\int_{-H}^H f(c+iv)\,dv\]
for any occurring integrable complex function $f$ without explicitly stating this every time.

\subsubsection{Exact factorization}

Defining $\epsi_T(g):=\frac{g-1}{g+1}e^{-2\gamma  T g},$ we rewrite
\begin{equation}\label{eq:sinhcosh-exact}
    \sinh(\gamma  T g)+g\cosh(\gamma  T g)=e^{\gamma  Tg}\frac12\left(1-e^{-2\gamma  Tg}+g(1+e^{-2\gamma  Tg})\right)
        =\frac{g+1}{2}e^{\gamma  T g}\left(1+\epsi_T(g)\right), 
\end{equation}
valid for all complex $g$. Further we define
\begin{equation}\label{eq:ax-def}
    a_x(u):=\frac{1}{u}\sqrt{\frac{2g(u)}{1+g(u)}}\,\exp\left(-\frac{\gamma}{2}x^2(g(u)-1)\right),
\end{equation}
\begin{equation}\label{eq:psi-def}
    \psi(u):=\frac{\gamma}{2}\,(1-g(u))+zu,
\end{equation}
as well as
\begin{equation}\label{eq:rho-def}
    \rho_T(u):=\left(1+\epsi_T(g(u))\right)^{-1/2}-1,
\end{equation}
\begin{equation}\label{eq:kappa-def}
    \kappa_T(u):=\exp\left(-\frac{\gamma}{2}x^2(g(u)-1)\left(\frac{1-e^{-2\gamma  T g(u)}}{1+\epsi_T(g(u))}-1\right)\right)-1.
\end{equation}
    
Using these auxiliary functions we can decompose the integrand in \eqref{eq:bromwich} using the explicit formula for \eqref{eq:LT}:
\begin{equation}\label{eq:exact-decomp}
    \frac{M_T(u)}{u}\,e^{uy}
    =a_x(u)\;\exp\left(T\psi(u)\right)\;\Big(1+\rho_T(u)\Big)\,\Big(1+\kappa_T(u)\Big).
\end{equation}
    
Using this decomposition, we can split \eqref{eq:bromwich} into the \emph{main} and \emph{error} integrals, for $H>0$,
\begin{align*}\label{eq:IandE}
    I_T(H)&:=\frac{1}{2\pi i}\int_{c-iH}^{c+iH} a_x(u)e^{T\psi(u)}\,du,\\
    E_T(H)&:=\frac{1}{2\pi i}\int_{c-iH}^{c+iH} a_x(u)e^{T\psi(u)}\left(\rho_T(u)+\kappa_T(u)+\rho_T(u)\kappa_T(u)\right)\,du.
\end{align*}
Then by \eqref{eq:bromwich} and \eqref{eq:exact-decomp},
\begin{equation}\label{eq:FT-splitH}
    \P_x(Z_T\leq y)=\lim_{H\to\infty}\left(I_T(H)+E_T(H)\right),
\end{equation}
where we can choose any $c>0$ for the path integrals.
    
\subsubsection{Error integral}
First, we note for later that $a_x(u)$ can be uniformly bounded on any Bromwich line $u=c+iv, c>0$:
\begin{align}\label{eq:axboundprelim}
    \abs a_x(u)\abs&=\frac{1}{\abs u\abs}\left(\frac{2}{\abs1+\frac{1}{g(u)}\abs}\right)^{1/2} \exp\left(-\frac{\gamma}{2}x^2(\Re g(u) - 1)\right)\nonumber\\
    &\leq \frac{1}{\Re u}\left(\frac{2}{1+\Re\frac{1}{g(u)}}\right)^{1/2} \exp\left(-\frac{\gamma}{2}x^2(\Re g(u) - 1)\right)\nonumber\\
    &\leq \frac{1}{c}\sqrt{2}\exp\left(-\frac{\gamma}{2}x^2(g(c) - 1)\right),
\end{align}
where we used $\Re\frac{1}{g(u)}\geq0$ (Lemma \ref{lem:reg}(e)) and $\Re g(u)\geq g(c)$ (Lemma~\ref{lem:reg}(a)). Since $g(c)>1$ (Lemma \ref{lem:reg}(a)) we have $\frac{2g(c)}{1+g(c)}>1$, and therefore \eqref{eq:axboundprelim} yields
\begin{align}\label{eq:axbound}
    \abs a_x(u)\abs
    \leq \sqrt{2}\,\frac{1}{c}\left(\frac{2g(c)}{1+g(c)}\right)^{1/2}\exp\left(-\frac{\gamma}{2}x^2(g(c) - 1)\right)
    =\sqrt{2}\,a_x(c).
\end{align}
Now, we identify the order of the terms in the error integral.
\begin{lem}\label{lem:errorint}
    Let $c>0$. There exists a $T_E>0$, such that for all $T\geq T_E
    $, all $\abs x\abs\leq\sqrt{T}$ and $y>0$, we have
    \begin{equation*}
        \limsup_{H\to\infty} \left|E_T(H)\right|\leq C\, a_x(c)\,T^{-2}\,e^{T\psi(c)},
    \end{equation*}
    where the constant $C>0$ is independent of $x,y,u,T$. 
\end{lem}
\begin{proof}
    Since $\Re g(u)>1$ by Lemma \ref{lem:reg}(a), we first get, for all $T\geq 1$,
    \begin{align}\label{eq:epsTbound}
        \abs\epsi_T(g(u))\abs=\frac{\abs g(u)-1\abs}{\abs g(u)+1\abs}\,e^{-2\gamma T\Re g(u)}
        =\left(\frac{(\Re g(u)-1)^2+(\Im g(u))^2}{(\Re g(u)+1)^2+(\Im g(u))^2}\right)^{1/2}\,e^{-2\gamma T\Re g(u)}\leq e^{-2\gamma T\Re g(u)}<e^{-2\gamma}.
    \end{align}
    This shows, that $\abs \epsi_T(g(u))\abs<e^{-2\gamma}<e^{-\gamma}<1$ for all $T\geq 1$, therefore Lemma \ref{lem:invsqrtTaylor} gives the existence of some constant $C_\rho>0$ (independent of $u,T$) such that\medskip
    \begin{equation} \label{eq:rhobound}
        \abs \rho_T(u)\abs=\abs (1+\epsi_T(g(u)))^{-1/2}-1\abs\leq C_\rho\,\abs\epsi_T(g(u))\abs\leq C_\rho\,e^{-2\gamma T\Re g(u)}.
    \end{equation}
    Turning our attention to $\kappa$, some rearranging yields
    \begin{align*}
        b_x(u):=&-\frac{\gamma}{2}x^2(g(u)-1)\left(\frac{1-e^{-2\gamma T g(u)}}{1+\epsi_T(g(u))}-1\right)\\
        =&-\frac{\gamma}{2}x^2(g(u)-1)\left(\frac{1-e^{-2\gamma T g(u)}-1-\frac{g(u)-1}{g(u)+1}e^{-2\gamma T g(u)}}{1+\epsi_T(g(u))}\right)\\
        =&~\frac{\gamma}{2}x^2(g(u)-1)\left(1+\frac{g(u)-1}{g(u)+1}\right)\,\frac{e^{-2\gamma T g(u)}}{1+\epsi_T(g(u))}\\
        =&~\frac{\gamma}{2}x^2\frac{g(u)-1}{g(u)+1}\,2g(u)\,\frac{e^{-2\gamma T g(u)}}{1+\epsi_T(g(u))}\\
        =&~\gamma x^2\frac{g(u)-1}{g(u)+1}\,g(u)\,\frac{e^{-2\gamma T g(u)}}{1+\epsi_T(g(u))}.
    \end{align*}
    As above, $\Re g(u)>1$ (Lemma \ref{lem:reg}(a))  implies $\left|\frac{g(u)-1}{g(u)+1}\right|\leq 1$ and since $\abs \epsi_T(g(u))\abs\leq e^{-2\gamma}$, by the inverse triangle inequality we have $\abs 1+\epsi_T(g(u))\abs\geq 1-e^{-2\gamma}$ for all $T\geq 1$. 
    Collecting these observations and using Lemma \ref{lem:reg}(b), we can bound $b_x(u)$ for all $T\geq 1$ by
    \begin{align*}\label{eq:kappaexpobound}
        \abs b_x(u)\abs&=\gamma x^2\left|\frac{g(u)-1}{g(u)+1}\right|\,\abs g(u)\abs\,\frac{e^{-2\gamma T\Re g(u)}}{\abs 1+\epsi_T(g(u))\abs}
        \leq \gamma x^2\sqrt{2}\,\Re g(u)\frac{e^{-2\gamma T\Re g(u)}}{1-e^{-2\gamma}}\\
        &\leq \frac{\sqrt{2}\gamma x^2}{1-e^{-2\gamma}}\frac{\Re g(u)}{\gamma T\Re g(u)}e^{-\gamma T\Re g(u)}\leq \frac{\sqrt{2}}{1-e^{-2\gamma}}e^{-\gamma T\Re g(u)},
    \end{align*}
    as $x^2\leq T$ and in the second to last inequality, we used $e^{-s}<\frac{1}{s}$ for $s>0$.
    This shows there exists some constant $C_b>0$ (independent of $x,y,u,T$) such that $\abs b_x(u)\abs\leq C_b\,e^{-\gamma T\Re g(u)}<C_b\,e^{-\gamma T}$ for all $T\geq 1$.
    Let $T_E>1$ be such that $e^{-\gamma T}<\frac{1}{2C_b}$ for all $T\geq T_E$.
    For $T\geq T_E$, we therefore have $\abs b_x(u)\abs<C_b\,e^{-\gamma T}<\frac12$ and Lemma \ref{lem:expTaylor} yields
    \begin{equation}\label{eq:kappabound}
        \abs\kappa_T(u)\abs=\left|e^{b_x(u)}-1\right|\leq \frac{e}{1-\frac12}\abs b_x(u)\abs\leq 2e\,C_b\,e^{-\gamma T\Re g(u)}:= C_\kappa \,e^{-\gamma T\Re g(u)},
    \end{equation}
    where $C_\kappa>0$ is a constant independent of $x,u,T$.\bigskip\\
    With the bounds \eqref{eq:axbound}, \eqref{eq:rhobound} and \eqref{eq:kappabound}, we can finally bound the error integral, for all $T\geq T_E$:
    \begin{align*}
        \left|E_T(H)\right|&\leq\frac{1}{2\pi}\int_{c-iH}^{c+iH}\abs a_x(u)\abs e^{T\Re\psi(u)}\left(\abs\rho_T(u)\abs+\abs\kappa_T(u)\abs+\abs\rho_T(u)\abs\,\abs\kappa_T(u)\abs\right)\,du\\
        &\leq \frac{\sqrt{2}\,a_x(c)}{2\pi}\int_{c-iH}^{c+iH} e^{T\frac{\gamma}{2}(1-\Re g(u))}\,e^{y c}\left(C_\rho\,e^{-2\gamma T\Re g(u)}+ C_\kappa \,e^{-\gamma T\Re g(u)}+C_\rho C_\kappa \,e^{-3\gamma T\Re g(u)}\right)\,du\\
        &\leq \frac{\sqrt{2}\,a_x(c)}{2\pi}\int_{c-iH}^{c+iH}  e^{T\frac{\gamma}{2}(1-\Re g(u))}\,e^{y c}\,C_{\rho,\kappa}\,e^{-\gamma T\Re g(u)}\,du,
    \end{align*}
    where $C_{\rho,\kappa}:=C_\rho+C_\kappa+C_\rho C_\kappa$. Using $\Re g(u)\geq g(c)$ (Lemma \ref{lem:reg}(a)) 
    we obtain, for all $T\geq T_E$,
    \begin{equation}
        \left|E_T(H)\right|\leq \frac{\sqrt{2}\,a_x(c)}{2\pi}e^{yc}e^{T\frac{\gamma}{2}\left(1-g(c)\right)}
        \int_{-H}^H  C_{\rho,\kappa} \,e^{-\gamma T\Re g(c+iv)}\,dv.
    \end{equation}
   We can evaluate the integral using Lemma \ref{lem:reg}(c) and Lemma \ref{lem:int}:
    \begin{align*}
        \int_{-H}^H e^{-\gamma  T\Re g(c+iv)}\,dv<~&\int_{-H}^H e^{-\gamma  T\frac{\sqrt{\abs v\abs}}{\gamma}}\,dv
        =2\int_0^H e^{-T\sqrt{v}}\,dv\\
        =~&\frac{4}{T^2}\left(1-e^{-T\sqrt{H}}\left(  T\sqrt{H}+1\right)\right)\overset{H\to\infty}{\longrightarrow}\frac{4}{T^2}.
    \end{align*}
    This shows, for all $T\geq T_E$,
    \begin{equation}\label{eq:errorintbound}
        \limsup_{H\to\infty} \left|E_T(H)\right| \leq \frac{\sqrt{2}\,a_x(c)}{2\pi}e^{y c}e^{T\frac{\gamma}{2}\left(1-g(c)\right)}\frac{4}{T^2} C_{\rho,\kappa}
        =C \frac{a_x(c)}{T^2} e^{T\psi(c)},
    \end{equation}
    where $C>0$ is a constant independent of $x,u,T$, which was the claim.
\end{proof}

\subsubsection{Main integral}
For all lemmas in this section we consider 
$u=c+iv$ with $c>0$ and $v\in\R$.
Further, we only suppose $x\in\R$ and $y>0$ with $y\sim\theta T$. We will first examine the behaviour of the main integral outside a fixed local segment:
\begin{lem}\label{lem:mainintnotlocalbound}
    For any $r>0$ and $T\geq 1$
    \begin{equation*}
        \limsup\limits_{H\to\infty}\left|I_T(H)-I_T(r)\right|\leq C_{c,r}\,a_x(c)\,\frac{1}{T}\,e^{T\Re\psi(c+ir)},    
    \end{equation*}
    where $C_{c,r}>0$ depends only on $c,r$.
\end{lem}
\begin{proof}
    By definition, for $H>r$
    \begin{equation}
        I_T(H)-I_T(r)=\frac{1}{2\pi}\int_{r<\abs v\abs\leq H}a_x(c+iv)e^{T\psi(c+iv)}\,dv.
    \end{equation}
    By \eqref{eq:axbound}, we get
    \begin{align}\label{eq:mainintnotlocal}
        \left|I_T(H)-I_T(r)\right|&\leq\frac{\sqrt{2}}{2\pi}\,a_x(c)\int_{r<\abs v\abs\leq H}e^{T\Re\psi(c+iv)}\,dv
        =\frac{\sqrt{2}}{\pi}\,a_x(c)\int_{r}^H e^{T\Re\psi(c+iv)}\,dv \nonumber\\
        &=\frac{\sqrt{2}}{\pi}\,a_x(c)\,e^{T\Re\psi(c+ir)}\int_{r}^H\,e^{T\Re\psi(c+iv)-T\Re\psi(c+ir)}\,dv \nonumber\\
        &=\frac{\sqrt{2}}{\pi}\,a_x(c)\,e^{T\Re\psi(c+ir)}\int_{r}^H\,e^{-T\frac{\gamma}{2}(\Re g(c+iv)-\Re g(c+ir))}\,dv, 
    \end{align}
    where we used in the second step, that the real part of the principal branch of the complex root is invariant under conjugation of the input (i.e.\ $\Re g(c+iv)=\Re g(c-iv)$), which implies that the same holds for $\psi$.
    We can analyze the last integral with Lemma \ref{lem:regimcomponent}. The lemma
    gives us the existence of a constant $K_{c,r}>0$ (not depending on $H$) such that
    \begin{equation}\label{eq:regnotlocalbound}
        \Re g(c+iv)-\Re g(c+ir)\geq K_{c,r}\left(\sqrt{v}\,-\sqrt{r}\right)\quad\text{for all}~v\geq r.
    \end{equation}
    Now, using \eqref{eq:regnotlocalbound}, the last integral in \eqref{eq:mainintnotlocal} can be bounded from above:
    \begin{align*}
        \int_{r}^H\,e^{-T\frac{\gamma}{2}\left(\Re g(c+iv)-\Re g(c+ir)\right)}\,dv
        \leq\int_{r}^H\,e^{-T\frac{\gamma}{2}K_{c,r}\left(\sqrt{v}\,-\sqrt{r}\right)}\,dv
        =e^{T\frac{\gamma}{2}K_{c,r}\sqrt{r}}\int_{r}^H\,e^{-T\frac{\gamma}{2}K_{c,r}\sqrt{v}}\,dv\\
        =e^{T\frac{\gamma}{2}K_{c,r}\sqrt{r}} \frac{8}{\gamma^2K_{c,r}^2\,T^2}
        \left(e^{-T\frac{\gamma}{2}K_{c,r}\sqrt{r}}\left(T\frac{\gamma}{2}K_{c,r}\sqrt{r}+1\right)-e^{-T\frac{\gamma}{2}K_{c,r}\sqrt{H}}\left(T\frac{\gamma}{2}K_{c,r}\sqrt{H}+1\right)\right)\\
        \longrightarrow e^{T\frac{\gamma}{2}K_{c,r}\sqrt{r}} \frac{8}{\gamma^2K_{c,r}^2\,T^2}
        e^{-T\frac{\gamma}{2}K_{c,r}\sqrt{r}}\left(T\frac{\gamma}{2}K_{c,r}\sqrt{r}+1\right)\\
        =\frac{1}{T}\left(\frac{4\sqrt{r}}{\gamma K_{c,r}}+\frac{8}{\gamma^2K_{c,r}^2\,T}\right)
        \leq \frac{1}{T}\left(\frac{4\sqrt{r}}{\gamma K_{c,r}}+\frac{8}{\gamma^2K_{c,r}^2}\right),
    \end{align*}
    for $H\to\infty$, where we evaluated the integral with Lemma \ref{lem:int} and used that $T\geq 1$.
    Together with \eqref{eq:mainintnotlocal}, this shows the claim:
    There exists a constant $C_{c,r}$ not depending on $T$ such that, for all $T\geq 1$,
    \[ 
        \limsup\limits_{H\to\infty}\left|I_T(H)-I_T(r)\right|\leq C_{c,r}\,a_x(c)\,\frac{1}{T}\,e^{T\Re\psi(c+ir)}    \qedhere
    \]
\end{proof}
    
To tackle the main integral inside a local segment, we expand $\psi$  around $c$. But first we note
\begin{equation}\label{eq:psiderivatives}
    \psi'(u)=\frac{\gamma}{2}\frac{-1}{2g(u)}\frac{2}{\gamma^2}+z=-\frac{1}{2\gamma\,g(u)} + z\quad\text{and}\quad\psi''(u)=\frac{1}{2\gamma\,g(u)^2}\frac{1}{g(u)}\frac{2}{\gamma^2}=\frac{1}{2\gamma^3\,g(u)^3}.
\end{equation}
Together with $g(c)>1$ (Lemma \ref{lem:reg}(a)) this also implies that $\psi''(c)>0$ and we can take roots of and divide by $\psi''(c)$ for any $c>0$,
a fact we will use throughout the proof.\\\\
From now on, let $T_0>0$ such that
$z\leq 2\theta<\frac{1}{\gamma}$ for all $T\geq T_0$.
$T_0$ exists as $z\to\theta$ and only depends on the specific rate of convergence of $y/\theta T$.
The condition $T\geq T_0$ ensures primarily that we can bound $\psi$ independently of $T$ for the remainder of the proof.
\begin{lem}\label{lem:psitaylorline}
    For each $c>0$ and $T\geq T_0$, we have
    \begin{equation}\label{eq:psitaylorline}
        \psi(c+iv)=\psi(c)+i\psi'(c)v-\frac{\psi''(c)}{2}v^2+R_T(c+iv)\quad\text{for}~\abs v\abs\leq c
    \end{equation}
    where 
    \begin{equation}\label{eq:psitaylorlinebound}
        \abs R_T(c+iv)\abs\leq \frac{3}{\gamma}\frac{\abs v\abs^3}{c^2}\frac{1}{1-\frac{\abs v\abs}{c}}.
    \end{equation}
\end{lem}
\begin{proof}
    Since $g$ is holomorphic on $\Re u\geq0$, 
    $\psi$ is holomorphic on $B_{c}(c)$. Therefore, there exists a Taylor expansion around $c$:
    \begin{equation}\label{eq:psitaylor}
        \psi(u)=\psi(c)+\psi'(c)(u-c)+\frac{\psi''(c)}{2}(u-c)^2+R_T(u),
    \end{equation}
    with (cf. Lemma \ref{lem:complextaylorestimates})
    \begin{equation}\label{eq:psitaylorbound}
        \abs R_{T}(u)\abs\leq \max_{s\in[0,2\pi]} \abs \psi\left(c+c e^{is}\right)\abs\frac{\abs u-c\abs^3}{{c}^3}\frac{1}{1-\frac{\abs u-c\abs}{c}}\leq \frac{3}{\gamma}\frac{\abs u-c\abs^3}{{c}^2}\frac{1}{1-\frac{\abs u-c\abs}{c}},
    \end{equation}
    for all $T \geq T_0$, because Lemma \ref{lem:reg}(f) yields for such $T$ that
    \begin{align*}
        \abs \psi\left(c+c e^{is}\right)\abs&\leq \frac{\gamma}{2}\left|1-g(c+c e^{is})\right|+z c\left| 1+e^{is}\right|\nonumber\\
        &\leq\frac{\gamma}{2}\frac{1}{\gamma^2}2c+2z c
        \leq 2c\left(\frac{1}{2\gamma}+\frac{1}{\gamma}\right)
        =\frac{3}{\gamma}c.
    \end{align*}
    On the Bromwich line through $c$ with $\abs v\abs\leq c$ \eqref{eq:psitaylor} and \eqref{eq:psitaylorbound} turn into the claimed formulae.
\end{proof}
    
Inside a local segment of a Bromwich line $c+iv,~\abs v\abs\leq r$, we can also control the main integral away from an asymptotically small neighbourhood of $c$. From here on out, let $r$ depend on $c$ in the following way:
\begin{equation}\label{eq:defr(c)}
    r=r(c):=\min\left\{\frac{c}{2},\frac{\gamma\,\psi''(c)\,c^2}{24}\right\}.   
\end{equation}
\begin{lem}\label{lem:IT(r-Talpha)}
    Let $\alpha>0$.
    For  $T\geq \max\left\{T_0,\,r^{-1/\alpha}\right\}$, we have
    \begin{equation*}\label{eq:IT(r0-Talpha)}
        \left|I_T(r)-I_T(T^{-\alpha})\right|\leq\frac{\sqrt{2}\,r}{\pi}\,a_x(c)\,e^{T\psi(c)}\,e^{-\frac{\psi''(c)}{4}T^{1-2\alpha}}.
    \end{equation*}
\end{lem}
\begin{proof}
    We will first see the reason for the definition of $r$:
    Since $\psi(c), \psi'(c)$ and $\psi''(c)$ are real (and therefore $i\psi'(c)v$ purely imaginary) and $T\geq T_0$, by Lemma \ref{lem:psitaylorline} we have
    \[\Re\psi(c+iv)=\psi(c)-\frac{\psi''(c)}{2}v^2+\Re R_T(c+iv)\]
    and
    \begin{align*}
        \Re R_T(c+iv)&\leq\left|R_T(c+iv)\right|
        \leq \frac{3}{\gamma} \frac{\abs v\abs^3}{{c}^2}\frac{1}{1-\frac{\abs v\abs}{c}}\\
        &\leq \frac{3}{\gamma}\frac{r}{{c}^2}\frac{1}{1-\frac{r}{c}}v^2
        \leq \frac{3}{\gamma} \frac{\gamma\,\psi''(c)\,c^2}{24\,{c}^2}\frac{1}{1-\frac{c}{2c}}v^2
        =\frac{\psi''(c)}{4}v^2,
    \end{align*}
    on $\abs v\abs\leq r$, where we used \eqref{eq:psitaylorlinebound} and $r\leq\frac{c}{2}$ as well as
    $r\leq\frac{\gamma\,\psi''(c)\,{c}^2}{24}$, by definition of $r$.
    This yields
    \begin{equation}\label{eq:repsibound}
        \Re\psi(c+iv)\leq \psi(c)-\frac{\psi''(c)}{2}v^2+\frac{\psi''(c)}{4}v^2
        =\psi(c)-\frac{\psi''(c)}{4}v^2
    \end{equation}
    on $\abs v\abs\leq r$.
    We use this and \eqref{eq:axbound} (as well as $r-T^{-\alpha}\leq r$) to see that
    \begin{align*}
        \left|I_T(r)-I_T(T^{-\alpha})\right|
        &\leq \frac{1}{2\pi}\int_{T^{-\alpha}<\abs v\abs\leq r}\left|a_x(c+iv)\right|\,e^{T\Re\psi(c+iv)}\,dv\\
        &\leq \frac{\sqrt{2}}{2\pi}\,a_x(c)\int_{T^{-\alpha}<\abs v\abs\leq r}e^{T\Re\psi(c+iv)}\,dv\\
        &\leq \frac{\sqrt{2}}{2\pi}\,a_x(c)\int_{T^{-\alpha}<\abs v\abs\leq r}e^{T\left(\psi(c)-\frac{\psi''(c)}{4}v^2\right)}\,dv\\
        &= \frac{\sqrt{2}}{\pi}\,a_x(c)\,e^{T\psi(c)}\int_{T^{-\alpha}}^{r}e^{-T\frac{\psi''(c)}{4}v^2}\,dv\\
        &\leq \frac{\sqrt{2}}{\pi}\,a_x(c)\,e^{T\psi(c)}\left(r-T^{-\alpha}\right)\,e^{-\frac{\psi''(c)}{4}T^{1-2\alpha}}\\
        &\leq \frac{\sqrt{2}\,r}{\pi}\,a_x(c)\,e^{T\psi(c)}\,e^{-\frac{\psi''(c)}{4}T^{1-2\alpha}}. \qedhere
    \end{align*}
\end{proof}

With Lemma \ref{lem:mainintnotlocalbound} and Lemma \ref{lem:IT(r-Talpha)} we control the asymptotic behaviour of the Bromwich integral along the whole Bromwich line except on an asymptotically vanishing neighbourhood of $c$.
The next two lemmas concern themselves with the behaviour of the integral inside this vanishing neighbourhood.   
We begin by examining the contribution of $a_x$.
\begin{lem}\label{lem:I_T(Talpha)1}
    Let $0<\beta<\alpha$ and take $r$ as defined in \eqref{eq:defr(c)}.
    Then there exists a constant $C>0$ only depending on $c$, such that
    for 
    \[T\geq \max\left\{T_0, \left(\frac{2}{r}\right)^{1/\alpha}, \abs x\abs^{2/\beta},\,\gamma^{1/(\beta-\alpha)}\right\},\] 
    we have
    \[\left|I_T(T^{-\alpha})-\frac{1}{2\pi}a_x(c)\int_{-T^{-\alpha}}^{T^{-\alpha}}e^{T\psi(c+iv)}\,dv\right|
    \leq C\,a_x(c)\,T^{\beta-1}\,e^{T\psi(c)}.\]
\end{lem}
\begin{proof}
    First of all, we rewrite $a_x(u)=a(u)\,e^{-\frac{\gamma}{2}x^2(g(u)-1)}$ with $a(u):=\frac{1}{u}\left(\frac{2g(u)}{1+g(u)}\right)^{1/2}$.\\
    With this, we now consider
    \begin{align*}
        &I_T^{(1)}:= \frac{1}{2\pi}\int_{-T^{-\alpha}}^{T^{-\alpha}}\left(a(c+iv)-a(c)\right)\,e^{-\frac{\gamma}{2}x^2(g(c+iv)-1)}\,e^{T\psi(c+iv)}\,dv,\\
        &I_T^{(2)}:= \frac{1}{2\pi}\int_{-T^{-\alpha}}^{T^{-\alpha}}a(c)\left(e^{-\frac{\gamma}{2}x^2(g(c+iv)-1)}-e^{-\frac{\gamma}{2}x^2(g(c)-1)}\right)\,e^{T\psi(c+iv)}\,dv,
    \end{align*}
    and notice that the difference we want to bound is precisely $I_T^{(1)}+I_T^{(2)}$.\\\\
    \textit{Step 1: Bound for $I_T^{(1)}$.} We note that $a(\cdot)$ is holomorphic away from $0$ and
    \[\abs a(u)\abs = \frac{1}{\abs u\abs}\left|\frac{2}{1+\frac{1}{g(u)}}\right|^{1/2}
    \leq \frac{1}{\Re u}\sqrt{2}\left(\frac{1}{1+\Re \frac{1}{g(u)}}\right)^{1/2}\leq\frac{\sqrt{2}}{\Re u},\]
    by Lemma \ref{lem:reg}(e).
    Therefore, the complex Taylor estimate (Lemma \ref{lem:complextaylorestimates} on $B_r(c)$) yields
    \begin{align}
        \left|a(c+iv)-a(c)\right|\leq\max_{\abs z-c\abs=r}\left|a(z)\right|\,\frac{\abs v\abs}{r}\frac{1}{1-\frac{\abs v\abs}{r}}
        \leq \max_{\abs z-c\abs=r} \frac{\sqrt{2}}{\Re z} \frac{1}{r-\abs v\abs} \abs v\abs
        \leq\frac{\sqrt{2}}{c-r}\frac{2}{r}\abs v\abs
        \leq \frac{4\sqrt{2}}{c\,r}\abs v\abs,
    \end{align}
    for any $\abs v\abs\leq T^{-\alpha}\leq \frac{r}{2}$ (as $T\geq (2/r)^{1/\alpha}$), and we used $r\leq \frac{c}{2}$ due to the definition of $r$.
    With this estimate, we have
    \begin{align}\label{eq:IT(1)}
        \left|I_T^{(1)}\right|
        &\leq \frac{1}{2\pi}\int_{-T^{-\alpha}}^{T^{-\alpha}}\left|a(c+iv)-a(c)\right|\,e^{-\frac{\gamma}{2}x^2(\Re g(c+iv)-1)}\,e^{T\Re\psi(c+iv)}\,dv\nonumber\\
        &\leq \frac{2\sqrt{2}}{\pi\,r}\frac{1}{c}\,e^{-\frac{\gamma}{2}x^2(g(c)-1)}\int_{-T^{-\alpha}}^{T^{-\alpha}}\left|v\right|\,e^{T\Re\psi(c+iv)}\,dv\nonumber\\
        &\leq \frac{2\sqrt{2}}{\pi\,r}\,a(c)\,e^{-\frac{\gamma}{2}x^2(g(c)-1)}\int_{-T^{-\alpha}}^{T^{-\alpha}}\left|v\right|\,e^{T\Re\psi(c+iv)}\,dv\nonumber\\
        &= \frac{2\sqrt{2}}{\pi\,r}\,a_x(c)\int_{-T^{-\alpha}}^{T^{-\alpha}}\left|v\right|\,e^{T\Re\psi(c+iv)}\,dv,
    \end{align}
    where we used $\Re g(c+iv)\geq g(c)>1$ (Lemma \ref{lem:reg}(a)) and $1<\frac{2g(c)}{1+g(c)}$ (implying $\frac{1}{c}\leq a(c)$).\\\\
    \textit{Step 2: Bound for $I_T^{(2)}$.} We first note, since $T^\beta\geq x^2$ and $T^{\beta-\alpha}\leq \gamma$,
    \begin{align*}
        \left|\frac{\gamma}{2}x^2\left(g(c+iv)-g(c)\right)\right|
        &=\frac{\gamma}{2}x^2\left|\frac{1+\frac{2}{\gamma^2}(c+iv)-\left(1+\frac{2}{\gamma^2}c\right)}{g(c+iv)+g(c)}\right|\\
        &\leq\frac{\gamma}{2}x^2\frac{\frac{2}{\gamma^2}\abs v\abs}{\Re g(c+iv)+g(c)}
        \leq \frac{x^2}{2\gamma}\abs v\abs\leq \frac{1}{2\gamma} T^{\beta-\alpha}\leq \frac12,
    \end{align*}
    for any $\abs v\abs\leq T^{-\alpha}$, where we used Lemma \ref{lem:reg}(a) to bound the denominator.     Lemma \ref{lem:expTaylor} therefore yields
    \begin{align*}
        \left|e^{-\frac{\gamma}{2}x^2\left(g(c+iv)-g(c)\right)}-1\right|
        \leq 2e\left|\frac{\gamma}{2}x^2\left(g(c+iv)-g(c)\right)\right|
        \leq \frac{e\,x^2}{\gamma}\abs v\abs.
    \end{align*}
    We use this estimate and the same observations as above to get
    \begin{align}\label{eq:IT(2)}
        \left|I_T^{(2)}\right|
        &\leq \frac{1}{2\pi}\,a(c)\int_{-T^{-\alpha}}^{T^{-\alpha}}\left|e^{-\frac{\gamma}{2}x^2(g(c+iv)-1)}-e^{-\frac{\gamma}{2}x^2(g(c)-1)}\right|\,e^{T\Re\psi(c+iv)}\,dv\nonumber\\
        &\leq \frac{1}{2\pi}\,a(c)\,e^{-\frac{\gamma}{2}x^2(g(c)-1)}\int_{-T^{-\alpha}}^{T^{-\alpha}}\left|e^{-\frac{\gamma}{2}x^2(g(c+iv)-g(c))}-1\right|\,e^{T\Re\psi(c+iv)}\,dv\nonumber\\
        &\leq \frac{e\,x^2}{2\gamma\,\pi}\,a_x(c)\int_{-T^{-\alpha}}^{T^{-\alpha}}\abs v\abs\,e^{T\Re\psi(c+iv)}\,dv.
    \end{align}
    \textit{Step 3: Evaluation of the remaining integral in \eqref{eq:IT(1)} and \eqref{eq:IT(2)}.}
    We note that, since $T\geq \left(\frac{2}{r}\right)^{1/\alpha}\geq \left(\frac{1}{r}\right)^{1/\alpha}$
    and $T\geq T_0$, the assumptions of Lemma \ref{lem:IT(r-Talpha)} are fulfilled.
    This means \eqref{eq:repsibound} holds as well and we get
    \begin{align}
        \int_{-T^{-\alpha}}^{T^{-\alpha}}\abs v\abs\,e^{T\Re\psi(c+iv)}\,dv
        &\leq \int_{-T^{-\alpha}}^{T^{-\alpha}}\abs v\abs\,e^{T\left(\psi(c)-\frac{\psi''(c)}{4}v^2\right)}\,dv
        =e^{T\psi(c)}\int_{-T^{-\alpha}}^{T^{-\alpha}}\abs v\abs\,e^{-T\frac{\psi''(c)}{4}v^2}\,dv\nonumber\\
        &=2e^{T\psi(c)}\int_{0}^{T^{-\alpha}}v\,e^{-T\frac{\psi''(c)}{4}v^2}\,dv
        =2e^{T\psi(c)}\frac{2}{T\psi''(c)}\left(\left.-e^{-T\frac{\psi''(c)}{4}v^2}\right|_0^{T^{-\alpha}}\right)\nonumber\\
        &=\frac{4}{\psi''(c)}\,T^{-1}\,e^{T\psi(c)}\left(1-e^{-\frac{\psi''(c)}{4}T^{1-2\alpha}}\right)
        \leq \frac{4}{\psi''(c)}\,T^{-1}\,e^{T\psi(c)}.
    \end{align}
    Inserting this into \eqref{eq:IT(1)} and \eqref{eq:IT(2)} yields the result:
    \begin{align*}
        \left|I_T^{(1)}+I_T^{(2)}\right|
        &\leq \left(\frac{2\sqrt{2}}{\pi\,r}+\frac{e\,x^2}{2\gamma\,\pi}\right)\,a_x(c)\int_{-T^{-\alpha}}^{T^{-\alpha}}\abs v\abs\,e^{T\Re\psi(c+iv)}\,dv\\
        &\leq \left(\frac{2\sqrt{2}}{\pi\,r}+\frac{e\,x^2}{2\gamma\,\pi}\right)\,a_x(c)\frac{4}{\psi''(c)}\,T^{-1}\,e^{T\psi(c)}\\
        &=\left(\frac{8\sqrt{2}}{\pi\,r\,\psi''(c)}\,T^{-1}+\frac{2e\,x^2}{\gamma\,\pi\,\psi''(c)}\,T^{-1}\right)\,a_x(c)\,e^{T\psi(c)}\\
        &\leq\left(\frac{8\sqrt{2}}{\pi\,r\,\psi''(c)}\,T^{-1}+\frac{2e}{\gamma\,\pi\,\psi''(c)}\,T^{\beta-1}\right)\,a_x(c)\,e^{T\psi(c)}\\
        &\leq\left(\frac{8\sqrt{2}}{\pi\,r\,\psi''(c)}+\frac{2e}{\gamma\,\pi\,\psi''(c)}\right)\,a_x(c)\,T^{\beta-1}\,e^{T\psi(c)}\\
        &=:C\,a_x(c)\,T^{\beta-1}\,e^{T\psi(c)}. \qedhere
    \end{align*}
\end{proof}

After controlling the behaviour of $a_x$, we now focus on $\psi$.
Specifically, this next lemma shows that it suffices to consider a quadratic expansion of $\psi$ around $c$, meaning any other contributions of $\psi$ along the Bromwich line to the integral are controlled.
\begin{lem}\label{lem:I_T(Talpha)2}
    Let $\alpha>\frac13$ and take $r$ as defined by \eqref{eq:defr(c)}.
    There exists a constant $C>0$ only depending on $c$, such that, for \[T\geq\max\left\{T_0,\,\left(\frac{2}{r}\right)^{1/\alpha},\,\left(\frac{2r}{\psi''(c)}\right)^{1/(1-3\alpha)}\right\},\] 
    we have 
    \begin{equation*}
        \left| \frac{1}{2\pi}a_x(c)\,\int_{-T^{-\alpha}}^{T^{-\alpha}} e^{T\psi(c+iv)}\,dv - \frac{1}{2\pi}a_x(c)\,\int_{-T^{-\alpha}}^{T^{-\alpha}}e^{T\left(\psi(c)+i\psi'(c)v-\frac{\psi''(c)}{2}v^2\right)}\,dv\right|
        \leq C\,a_x(c)\,T^{\frac12-3\alpha}
    \,e^{T\psi(c)}.
    \end{equation*}
\end{lem}
\begin{proof}
    We use the expansion of $\psi$ around $c$ (Lemma \ref{lem:psitaylorline}).
    For this, we first note that by definition of $r$ and the assumption on $T$, it holds $T^{-\alpha}\leq \frac{r}{2} < r\leq \frac{c}{2}$.
    This implies, as $T\geq T_0$, that for $\abs v\abs\leq T^{-\alpha}\leq \frac{c}{2}$, we have
    \begin{equation}\label{eq:TRTbound}
        \abs TR_T(c+iv)\abs\leq  T \frac{3}{\gamma} \frac{\abs v\abs^3}{c^2}\frac{1}{1-\frac{\abs v\abs}{c}}
        \leq T \frac{6}{\gamma\,c^2}\abs v\abs^3
        \leq T\frac{\psi''(c)}{4r}\abs v\abs^3
        \leq \frac{\psi''(c)}{4r}T^{1-3\alpha}\leq \frac12,
    \end{equation}
    where we used that, by definition of $r$ and the assumptions on $T$,
    \begin{equation}
        \frac{6}{\gamma\,c^2}\leq \frac{\psi''(c)}{4r}\quad\text{and}\quad T^{1-3\alpha}\leq\frac{2r}{\psi''(c)}.
    \end{equation}
    Lemma \ref{lem:expTaylor} therefore yields
    \begin{equation}
        \abs e^{TR_T(c+iv)}-1\abs\leq 2e\,\abs TR_T(c+iv)\abs\leq\frac{e\,\psi''(c)}{2r}T^{1-3\alpha}.
    \end{equation}
    Using this and the Taylor expansion of $\psi$ (Lemma \ref{lem:psitaylorline}) in the integral, we get the result:
    \begin{align*}
        &\left|\frac{1}{2\pi}a_x(c)\int_{-T^{-\alpha}}^{T^{-\alpha}} e^{T\psi(c+iv)}-e^{T\left(\psi(c)+i\psi'(c)v-\frac{\psi''(c)}{2}v^2\right)}\,dv\right|\\
        \leq~&\frac{a_x(c)}{2\pi}\int_{-T^{-\alpha}}^{T^{-\alpha}}\,e^{T\Re\left(\psi(c)+i\psi'(c)v-\frac{\psi''(c)}{2}v^2\right)}\left|e^{TR_T(c+iv)}-1\right|\,dv\\
        =~&\frac{a_x(c)}{2\pi}\int_{-T^{-\alpha}}^{T^{-\alpha}}\,e^{T\left(\psi(c)-\frac{\psi''(c)}{2}v^2\right)}\left|e^{TR_T(c+iv)}-1\right|\,dv\\
        \leq~&\frac{a_x(c)}{2\pi}\frac{e\,\psi''(c)}{2r}T^{1-3\alpha}\,e^{T\psi(c)}\int_{-T^{-\alpha}}^{T^{-\alpha}}\,e^{-T\frac{\psi''(c)}{2}v^2}\,dv\\
        =~&\frac{a_x(c)}{2\pi}\frac{e\,\psi''(c)}{2r}T^{1-3\alpha}e^{T\psi(c)}\frac{1}{\sqrt{T\,\psi''(c)}}\int_{-\sqrt{\psi''(c)}\,T^{1/2-\alpha}}^{\sqrt{\psi''(c)}\,T^{1/2-\alpha}}\,e^{-\frac{s^2}{2}}\,ds,\\
        \leq~&\frac{a_x(c)}{2\pi}\frac{e\,\sqrt{\psi''(c)}}{2r}T^{\frac12-3\alpha}e^{T\psi(c)}\int_{-\infty}^{\infty}\,e^{-\frac{s^2}{2}}\,ds\\
        =~&\frac{a_x(c)}{\sqrt{2\pi}}\frac{e\,\sqrt{\psi''(c)}}{2r}T^{\frac12-3\alpha}\,e^{T\psi(c)}
        =:C\,a_x(c)\,T^{\frac12-3\alpha}\,e^{T\psi(c)}. \qedhere
    \end{align*}
\end{proof}

\subsubsection{Asymptotics}

We can now combine the previous lemmas in order to prove Theorem \ref{thm:uniformasymptotics}.
For this purpose, let $\overline{c}:=\frac{1-4\gamma^2\theta^2}{8\theta^2}$ and $c^\ast_T:=\frac{1-4\gamma^2z^2}{8z^2}$. 
The idea is to use $\overline{c}$ in the Bromwich inversion formula and then use $c^\ast_T$ to identify the correct asymptotic behaviour of that formula.
Similar to saddlepoint methods, the concrete choice of $c^\ast_T$ minimizes $\psi$ and therefore the second order Taylor approximation of $\psi$ around $c^\ast_T$ is purely real on the corresponding Bromwich line, which simplifies the evaluation of the integral, as we will see below (cf.\ \eqref{eq:asymptoticexact}).
Note that $c^\ast_T\to\overline{c}$ since $z\to\theta$ for $T\to\infty$.
Now fix an $\alpha>\frac13$ with $\beta<\alpha<\frac12$ (which is possible since $\beta<\frac12$ is assumed).
Note that, if $\abs y-\theta T\abs\leq KT^\beta$, we also have $\abs z-\theta\abs\leq KT^{\beta-1}$.
Therefore $\beta<\frac12$ implies the existence of some $T^{(0)}>0$ such that for all $T\geq T^{(0)}$ we have
\begin{equation}\label{eq:zthetabound}
    \frac{\theta}{2}<\theta-KT^{\beta-1}\leq z\leq \theta+KT^{\beta-1}\leq\frac{3\theta}{2}.
\end{equation}
Comparing the implicit definitions of $T^{(0)}$ and $T_0$ from the previous section, we see that we can choose $T^{(0)}\geq T_0$.
This means the condition $T\geq T_0$ in the lemmas of the previous section is trivially fulfilled for the rest of the proof (as we only consider $T\geq T^{(0)}$ from here on).\\\\  
Further, we have $g(\overline{c})=\sqrt{1+\frac{2}{\gamma^2}\frac{1-4\gamma^2\theta^2}{8\theta^2}}=\frac{1}{2\gamma\theta}$ and analogously $g(c^\ast_T)=\frac{1}{2\gamma z}$.
With these observations and \eqref{eq:zthetabound} it follows that
\begin{align}\label{eq:psidiffbound}
    \left|\psi(\overline{c})-\psi(c^\ast_T)\right|
    &=\left|\frac{\gamma}{2}\left(1-g(\overline{c})\right)+z\overline{c}-\frac{\gamma}{2}\left(1-g(c^\ast_T)\right)-zc^\ast_T\right|
    =\left|\frac{\gamma}{2}\left(g(c^\ast_T)-g(\overline{c})\right)+z(\overline{c}-c^\ast_T)\right|\nonumber\\
    &=\left|\frac{\gamma}{2}\left(\frac{1}{2\gamma z}-\frac{1}{2\gamma\theta}\right)+\frac{z}{8\theta^2z^2}\left(z^2-4\gamma^2\theta^2z^2-\theta^2+4\gamma^2\theta^2z^2\right)\right|\nonumber\\
    &=\left| \frac{1}{4\theta z}(\theta-z)+\frac{1}{8\theta^2 z}\left(z^2-\theta^2\right)\right|
    =\left| \frac{1}{8\theta^2 z}\left(2\theta^2-2\theta z+z^2-\theta^2\right)\right|\nonumber\\
    &=\frac{1}{8\theta^2 z}\left| z-\theta\right|^2
    \leq \frac{1}{8\theta^2}\frac{2}{\theta} K^2 T^{2\beta-2}=\frac{K^2}{4\theta^3} T^{2\beta-2},
\end{align}
for all $T\geq T^{(0)}$, as well as (by \eqref{eq:psiderivatives})
\begin{equation}\label{eq:psiderivativebound}
    \left|\psi'(\overline{c})\right|=\left| -\frac{1}{2\gamma g(\overline{c})}+z\right|=\left| z-\theta\right|\leq KT^{\beta-1}.
\end{equation}
For the subsequent arguments we keep in mind that $\psi''(\overline{c})>0$, since $\theta<\frac{1}{2\gamma}$ implies $\overline{c}>0$ (cf. \eqref{eq:psiderivatives}).
Set $\overline{r}:=r(\overline{c})=\min\left\{\frac{\overline{c}}{2},\frac{\gamma\,\psi''(\overline{c})\,\overline{c}^2}{24}\right\}$
(cf.\ \eqref{eq:defr(c)} in the previous section).\\

We will now use the Bromwich inversion \eqref{eq:bromwich} 
and \eqref{eq:FT-splitH} as well as the lemmas of the previous section to decompose $\P_x(Z_T\leq y)$ into parts that we bound separately:
\begin{equation}\label{eq:FTsplitDis}
    \left|\P_x(Z_T\leq y)-\frac{1}{\sqrt{2\pi}}\frac{a_x(\overline{c})}{\sqrt{\psi''(\overline{c})}}T^{-1/2}e^{T\psi(c^\ast_T)}\right|
    \leq D^{(1)}+D^{(2)}+D^{(3)}+D^{(4)}+D^{(5)}+D^{(6)}+ D^{(7)},
\end{equation}
where
\begin{align}
    D^{(1)}&:=\limsup\limits_{H\to\infty}\left|I_T(H)-I_T(\overline{r})\right|,\nonumber\\
    D^{(2)}&:=\left|I_T(\overline{r})-I_T(T^{-\alpha})\right|,\nonumber\\
    D^{(3)}&:=\left|I_T(T^{-\alpha})-\frac{1}{2\pi}a_x(\overline{c})\int_{-T^{-\alpha}}^{T^{-\alpha}}e^{T\psi(\overline{c}+iv)}\,dv\right|,\nonumber\\
    D^{(4)}&:=\left|\frac{1}{2\pi}a_x(\overline{c})\int_{-T^{-\alpha}}^{T^{-\alpha}}e^{T\psi(\overline{c}+iv)}\,dv-\frac{1}{2\pi}a_x(\overline{c})\int_{-T^{-\alpha}}^{T^{-\alpha}}e^{T\left(\psi(\overline{c})+i\psi'(\overline{c})v-\frac{\psi''(\overline{c})}{2}v^2\right)}\,dv\right|,\nonumber\\
    D^{(5)}&:=\left|\frac{1}{2\pi}a_x(\overline{c})\int_{-T^{-\alpha}}^{T^{-\alpha}}e^{T\left(\psi(\overline{c})+i\psi'(\overline{c})v-\frac{\psi''(\overline{c})}{2}v^2\right)}\,dv-\frac{1}{2\pi}a_x(\overline{c})\int_{-T^{-\alpha}}^{T^{-\alpha}}e^{T\left(\psi(c^\ast_T)-\frac{\psi''(\overline{c})}{2}v^2\right)}\,dv\right|,\nonumber\\
    D^{(6)}&:=\left|\frac{1}{2\pi}a_x(\overline{c})\int_{-T^{-\alpha}}^{T^{-\alpha}}e^{T\left(\psi(c^\ast_T)-\frac{\psi''(\overline{c})}{2}v^2\right)}\,dv-\frac{1}{\sqrt{2\pi}}\frac{a_x(\overline{c})}{\sqrt{\psi''(\overline{c})}}T^{-1/2}e^{T\psi(c^\ast_T)}\right|,\nonumber\\
    D^{(7)}&:=\limsup\limits_{H\to\infty}\left|E_T(H)\right|.
\end{align}
Later on we will prove that each of the $D^{(i)}$ is negligible. 
\begin{lem}\label{lem:Difirelation}
    For each $i=1,\ldots,7$ there exist a $f^{(i)}=f^{(i)}_{\beta,K}(T)\geq 0$ and a $T^{(i)}>T^{(0)}$ such that,
    for all $T\geq T^{(i)}$ and any $x\in\R$, $y>0$ with
    \[\abs x\abs\leq T^{\beta/2}\quad\text{and}\quad\abs y-\theta T\abs\leq KT^\beta,\]
    it holds
    \begin{equation}\label{eq:Difirelation}
        D^{(i)}\leq f^{(i)}\,a_x(\overline{c})\,T^{-1/2}\,e^{T\psi(c^\ast_T)},
    \end{equation}
    and $f^{(i)}$ vanishes for $T\to\infty$.
\end{lem}
With this lemma, \eqref{eq:FTsplitDis} turns into
\begin{equation}\label{eq:FTsplitfis}
    \left|\P_x(Z_T\leq y)-\frac{1}{\sqrt{2\pi}}\frac{a_x(\overline{c})}{\sqrt{\psi''(\overline{c})}}T^{-1/2}e^{T\psi(c^\ast_T)}\right|
    \leq\left(f^{(1)}+\ldots+f^{(7)}\right)\,a_x(\overline{c})\,T^{-1/2}\,e^{T\psi(c^\ast_T)},
\end{equation}
for all $T\geq\max\left\{T^{(1)},\ldots,T^{(7)}\right\}$ and any $x,y$ as specified in the lemma. Since the sum of the $f^{(i)}$ vanishes (in $T$) as well,
there has to be some $T_{\delta,\beta,K}\geq \max\left\{T^{(1)},\ldots,T^{(7)}\right\}$ (the dependencies will become apparent when we determine the $f^{(i)}$) such that, for all $T\geq T_{\delta,\beta, K}$
\begin{equation*}
    f^{(1)}+\ldots+f^{(7)}
    \leq \delta\,\frac{1}{\sqrt{2\pi}}\frac{1}{\sqrt{\psi''(\overline{c})}},
\end{equation*}
and plugging this into \eqref{eq:FTsplitfis} yields
\begin{equation}
    \left|\P_x(Z_T\leq y)-\frac{1}{\sqrt{2\pi}}\frac{a_x(\overline{c})}{\sqrt{\psi''(\overline{c})}}T^{-1/2}e^{T\psi(c^\ast_T)}\right|
    \leq \delta\times\frac{1}{\sqrt{2\pi}}\frac{a_x(\overline{c})}{\sqrt{\psi''(\overline{c})}}T^{-1/2}e^{T\psi(c^\ast_T)},
\end{equation}
or equivalently
\begin{equation}\label{eq:resultnoteval}
    (1-\delta)\frac{1}{\sqrt{2\pi}}\,\frac{a_x(\overline{c})}{\sqrt{\psi''(\overline{c})}}\,T^{-1/2}\,e^{T\psi(c^\ast_T)}
    \leq \P_x(Z_T\leq y)
    \leq (1+\delta)\frac{1}{\sqrt{2\pi}}\,\frac{a_x(\overline{c})}{\sqrt{\psi''(\overline{c})}}\,T^{-1/2}\,e^{T\psi(c^\ast_T)},
\end{equation}
for all $T\geq T_{\delta,\beta,K}$, any $\abs x\abs\leq T^{\beta/2}$ and any $y>0$ with $\abs y-\theta T\abs\leq KT^\beta$. Explicitly evaluating the terms in \eqref{eq:resultnoteval} in the next section then proves Theorem \ref{thm:uniformasymptotics}.\\\\
Before evaluating we still have to prove Lemma \ref{lem:Difirelation}.
\begin{proof}[Proof of Lemma~\ref{lem:Difirelation}]
    We will first identify the $f^{(i)}$ satisfying \eqref{eq:Difirelation} for all $i=1,\ldots,7$
    and then later confirm that they are non-negative and indeed vanishing for $T\to\infty$.\\\\
    \textit{Main integral outside a local segment} ($D^{(1)}$ \textit{and} $D^{(2)}$).\\
    Lemma \ref{lem:mainintnotlocalbound} gives us a constant $C^{(1)}>0$ only depending on $\overline{c}$
    (since $\overline{r}$ only depends on $\overline{c}$ as well)
    such that, for all $T\geq T^{(1)}=\max\{T^{(0)}, 1\}$,
    \begin{align}\label{eq:I_T(H-r)boundexact}
        D^{(1)}=\limsup\limits_{H\to\infty}\left|I_T(H)-I_T(\overline{r})\right|
        &\leq C^{(1)}\,a_x(\overline{c})\,T^{-1}\,e^{T\Re\psi(\overline{c}+ir)}\nonumber\\
        &\leq C^{(1)}\,T^{-1/2}\,a_x(\overline{c})\,T^{-1/2}\,e^{T\psi(\overline{c})}\nonumber\\
        &= C^{(1)}\,T^{-1/2} \,e^{T(\psi(\overline{c})-\psi(c^\ast_T))}\,a_x(\overline{c})\,T^{-1/2}\,e^{T\psi(c^\ast_T)}\nonumber\\
        &\leq C^{(1)}\,T^{-1/2} \,e^{\frac{K^2}{4\theta^3}T^{2\beta -1}}\cdot a_x(\overline{c})\,T^{-1/2}\,e^{T\psi(c^\ast_T)}\nonumber\\
        &=: f^{(1)}_{\beta,K}(T)\cdot a_x(\overline{c})\,T^{-1/2}\,e^{T\psi(c^\ast_T)}
    \end{align}
    where we used \eqref{eq:psidiffbound} (as $T\geq T^{(0)}$) and
    \begin{equation*}
        \Re \psi(c+iv)=\frac{\gamma}{2}(1-\Re g(c+iv))+zc
        \leq \frac{\gamma}{2}(1-g(c))+zc = \psi(c),
    \end{equation*}
    as $\Re g(c+iv)\geq g(c)$ (Lemma \ref{lem:reg}(a)) for any $c>0$ and $v\in\R$.     Setting $T^{(2)}:=\max\left\{T^{(0)},\left(1/{\overline{r}}\right)^{1/\alpha}\right\}$,
    Lemma \ref{lem:IT(r-Talpha)} together with \eqref{eq:psidiffbound} yields
    \begin{align}\label{eq:I_T(r-Talpha)boundexact}
        D^{(2)}=\left|I_T(\overline{r})-I_T(T^{-\alpha})\right|
        &\leq \frac{\sqrt{2}\,\overline{r}}{\pi}\,a_x(\overline{c})\,e^{-\frac{\psi''(\overline{c})}{4}T^{1-2\alpha}}\,e^{T\psi(\overline{c})}\nonumber\\
        &=\frac{\sqrt{2}\,\overline{r}}{\pi}\,T^{1/2}\,e^{-\frac{\psi''(\overline{c})}{4}T^{1-2\alpha}}\,e^{T\left(\psi(\overline{c})-\psi(c^\ast_T)\right)}\,a_x(\overline{c})\,T^{-1/2}\,e^{T\psi(c^\ast_T)}\nonumber\\
        &\leq \frac{\sqrt{2}\,\overline{r}}{\pi}\,T^{1/2}\,e^{-\frac{\psi''(\overline{c})}{4}T^{1-2\alpha}}\,e^{\frac{K^2}{4\theta^3}T^{2\beta-1}}\cdot a_x(\overline{c})\,T^{-1/2}\,e^{T\psi(c^\ast_T)}\nonumber\\
        &=:f^{(2)}_{\beta,K}(T)\cdot a_x(\overline{c})\,T^{-1/2}\,e^{T\psi(c^\ast_T)},
    \end{align}
    for all $T\geq T^{(2)}$.\\\\
    \textit{Main integral inside a local segment} ($D^{(3)}$ \textit{and} $D^{(4)}$).\\
    Take $T^{(3)}:=\max\left\{T^{(0)},\left(2/\overline{r}\right)^{1/\alpha},\,\gamma^{1/(\beta-\alpha)}\right\}$. Then, Lemma \ref{lem:I_T(Talpha)1} gives the existence of a constant $C^{(3)}>0$ only depending on $\overline{c}$,
    such that for any $T\geq T^{(3)}$ and any $\abs x\abs\leq T^{\beta/2}$ (being equivalent to $T\geq\abs x\abs^{2/\beta}$) we have
    \begin{align}\label{eq:I_T(Talpha)axboundexact}
        D^{(3)}=\left|I_T(T^{-\alpha})-\frac{1}{2\pi}a_x(\overline{c})\int_{-T^{-\alpha}}^{T^{-\alpha}}e^{T\psi(\overline{c}+iv)}\,dv\right|
        &\leq~C^{(3)}\,a_x(\overline{c})\,T^{\beta-1}\,e^{T\psi(\overline{c})}\nonumber\\
        &=C^{(3)}\,T^{\beta-1/2}\,e^{T\left(\psi(\overline{c})-\psi(c^\ast_T)\right)}\,a_x(\overline{c})\,T^{-1/2}\,e^{T\psi(c^\ast_T)}\nonumber\\
        &\leq C^{(3)}\,T^{\beta-1/2}\,e^{\frac{K^2}{4\theta^3}T^{2\beta-1}}\cdot a_x(\overline{c})\,T^{-1/2}\,e^{T\psi(c^\ast_T)}\nonumber\\
        &=:f^{(3)}_{\beta,K}(T)\cdot a_x(\overline{c})\,T^{-1/2}\,e^{T\psi(c^\ast_T)},
    \end{align}
    where the last inequality follows from \eqref{eq:psidiffbound}.
    By Lemma \ref{lem:I_T(Talpha)2} (using $\alpha>\frac13$) there exists a constant $C^{(4)}>0$ only depending on $\overline{c}$ such that 
    for all $T\geq T^{(4)}:=\max\left\{T^{(0)},\left(2/\overline{r}\right)^{1/\alpha},\,\left(2\overline{r}/{\psi''(\overline{c})}\right)^{1/(1-3\alpha)}\right\}$ it holds
    \begin{align}\label{eq:I_T(Talpha)psiboundexact}
        D^{(4)}&=\left|\frac{1}{2\pi}a_x(\overline{c})\int_{-T^{-\alpha}}^{T^{-\alpha}}e^{T\psi(\overline{c}+iv)}\,dv-\frac{1}{2\pi}a_x(\overline{c})\int_{-T^{-\alpha}}^{T^{-\alpha}}e^{T\left(\psi(\overline{c})+i\psi'(\overline{c})v-\frac{\psi''(\overline{c})}{2}v^2\right)}\,dv\right|\nonumber\\
        &\leq C^{(4)}\,a_x(\overline{c})\,T^{1/2-3\alpha}\,e^{T\psi(\overline{c})}\nonumber\\
        &=C^{(4)}\,T^{1-3\alpha}\,e^{T\left(\psi(\overline{c})-\psi(c^\ast_T)\right)}\,a_x(\overline{c})\,T^{-1/2}\,e^{T\psi(c^\ast_T)}\nonumber\\
        &\leq C^{(4)}\,T^{1-3\alpha}\,e^{\frac{K^2}{4\theta^3}T^{2\beta-1}}\cdot a_x(\overline{c})\,T^{-1/2}\,e^{T\psi(c^\ast_T)}\nonumber\\
        &=:f^{(4)}_{\beta,K}(T)\cdot a_x(\overline{c})\,T^{-1/2}\,e^{T\psi(c^\ast_T)},
    \end{align}
    where we used \eqref{eq:psidiffbound} for the last inequality.\\\\
    \textit{Evaluating the integral} ($D^{(5)}$ \textit{and} $D^{(6)}$).\\
    In order to find $f^{(5)}$, we first note that
    combining \eqref{eq:psidiffbound} and \eqref{eq:psiderivativebound} yields
    \begin{equation}
        \abs T\left(\psi(\overline{c})-\psi(c^\ast_T)+i\psi'(\overline{c})v\right)\abs
        \leq T\left(\frac{K^2}{4\theta^3} T^{2\beta-2}+KT^{\beta-1} \abs v\abs\right)
        \leq \frac{K^2}{4\theta^3} T^{2\beta-1}+KT^{\beta-\alpha},
    \end{equation}
    for all $T\geq T^{(0)}$ and $\abs v\abs\leq T^{-\alpha}$.
    As $\beta<\alpha<\frac12$ there exists a $T^{(5)}\geq T^{(0)}$ such that
    \[\abs T\left(\psi(\overline{c})-\psi(c^\ast_T)+i\psi'(\overline{c})v\right)\abs\leq\frac12,\] 
    and therefore, by Lemma \ref{lem:expTaylor},
    \begin{equation}\label{eq:exppsidiffbound}
        \abs e^{T\left(\psi(\overline{c})-\psi(c^\ast_T)+i\psi'(\overline{c})v\right)}-1\abs\leq 2e\abs T\left(\psi(\overline{c})-\psi(c^\ast_T)+i\psi'(\overline{c})v\right)\abs\leq 2e\left(\frac{K^2}{4\theta^3} T^{2\beta-1}+KT^{\beta-\alpha}\right),
    \end{equation}
    on $\abs v\abs\leq T^{-\alpha}$ for all $T\geq T^{(5)}$.
    We can now use \eqref{eq:exppsidiffbound} in the following:
    \begin{align}\label{eq:D5boundexact}
        D^{(5)}&=\left|\frac{a_x(\overline{c})}{2\pi}\int_{-T^{-\alpha}}^{T^{-\alpha}}e^{T\left(\psi(\overline{c})+i\psi'(\overline{c})v-\frac{\psi''(\overline{c})}{2}v^2\right)}\,dv-\frac{a_x(\overline{c})}{2\pi}\int_{-T^{-\alpha}}^{T^{-\alpha}}e^{T\left(\psi(c^\ast_T)-\frac{\psi''(\overline{c})}{2}v^2\right)}\,dv\right|\nonumber\\
        &=\frac{a_x(\overline{c})}{2\pi}\left|e^{T\psi(c^\ast_T)}\int_{-T^{-\alpha}}^{T^{-\alpha}}\left(e^{T\left(\psi(\overline{c})-\psi(c^\ast_T)+i\psi'(\overline{c})v\right)}-1\right)e^{-T\frac{\psi''(\overline{c})}{2}v^2}\,dv\right|\nonumber\\
        &\leq\frac{a_x(\overline{c})}{2\pi}\,e^{T\psi(c^\ast_T)}\int_{-T^{-\alpha}}^{T^{-\alpha}}\left|e^{T\left(\psi(\overline{c})-\psi(c^\ast_T)+i\psi'(\overline{c})v\right)}-1\right|e^{-T\frac{\psi''(\overline{c})}{2}v^2}\,dv\nonumber\\
        &\leq\frac{e}{\pi}a_x(\overline{c})\left(\frac{K^2}{4\theta^3} T^{2\beta-1}+KT^{\beta-\alpha}\right)\,e^{T\psi(c^\ast_T)}\int_{-T^{-\alpha}}^{T^{-\alpha}}e^{-T\frac{\psi''(\overline{c})}{2}v^2}\,dv\nonumber\\
        &=\frac{e}{\pi}a_x(\overline{c})\left(\frac{K^2}{4\theta^3} T^{2\beta-1}+KT^{\beta-\alpha}\right)\,\frac{1}{\sqrt{T\psi''(\overline{c})}}\,e^{T\psi(c^\ast_T)}\int_{-\sqrt{\psi''(\overline{c})}T^{1/2-\alpha}}^{\sqrt{\psi''(\overline{c})}T^{1/2-\alpha}}e^{-\frac{s^2}{2}}\,ds\nonumber\\
        &\leq \frac{e\,\sqrt{2}}{\sqrt{\pi\,\psi''(\overline{c})}}\left(\frac{K^2}{4\theta^3} T^{2\beta-1}+KT^{\beta-\alpha}\right)\,a_x(\overline{c})\,T^{-1/2}\,e^{T\psi(c^\ast_T)}\nonumber\\
        &=:f^{(5)}_{\beta,K}(T)~~a_x(\overline{c})\,T^{-1/2}\,e^{T\psi(c^\ast_T)},
    \end{align}
    for all $T\geq T^{(5)}$, where we used the substitution $s=\sqrt{T\,\psi''(\overline{c})}\,v$ and that
    \[\int_{-\sqrt{\psi''(\overline{c})}T^{1/2-\alpha}}^{\sqrt{\psi''(\overline{c})}T^{1/2-\alpha}}e^{-\frac{s^2}{2}}\,ds\leq \int_{-\infty}^{\infty}e^{-\frac{s^2}{2}}\,ds=\sqrt{2\pi}.\]
    With the same substitution, we also have
    \begin{align*}
        &\left|\frac{1}{\sqrt{2\pi}}\,\sqrt{T\,\psi''(\overline{c})}\int_{-T^{-\alpha}}^{T^{-\alpha}}e^{-T\frac{\psi''(\overline{c})}{2}v^2}\,dv-1\right|
        =\left|\frac{1}{\sqrt{2\pi}}\int_{-\sqrt{\psi''(\overline{c})}\,T^{1/2-\alpha}}^{\sqrt{\psi''(\overline{c})}\,T^{1/2-\alpha}}e^{-\frac{s^2}{2}}\,ds-1\right|\\
        =~&\frac{1}{\sqrt{2\pi}}\int_{\abs s\abs\geq\sqrt{\psi''(\overline{c})}\,T^{1/2-\alpha}} e^{-\frac{s^2}{2}}\,ds
        \leq \frac{\sqrt{2}}{\sqrt{\pi}\sqrt{\psi''(\overline{c})}}\,T^{\alpha-1/2},
    \end{align*}
    where we used Markov's inequality. This however implies for all $T\geq T^{(6)}:=T^{(0)}$
    \begin{align}\label{eq:asymptoticexact}
        D^{(6)}&=\left|\frac{1}{2\pi}a_x(\overline{c})\int_{-T^{-\alpha}}^{T^{-\alpha}}e^{T\left(\psi(c^\ast_T)-\frac{\psi''(\overline{c})}{2}v^2\right)}\,dv-\frac{1}{\sqrt{2\pi}}\,\frac{a_x(\overline{c})}{\sqrt{\psi''(\overline{c})}}\,T^{-1/2}\,e^{T\psi(c^\ast_T)}\right|\nonumber\\
        &=\frac{1}{\sqrt{2\pi}}\,\frac{a_x(\overline{c})}{\sqrt{\psi''(\overline{c})}}\,T^{-1/2}\,e^{T\psi(c^\ast_T)}\left|\frac{1}{\sqrt{2\pi}}\,\sqrt{T\,\psi''(\overline{c})}\int_{-T^{-\alpha}}^{T^{-\alpha}}e^{-T\frac{\psi''(\overline{c})}{2}v^2}\,dv-1\right|\nonumber\\
        &\leq \frac{1}{\pi\,\psi''(\overline{c})}\,T^{\alpha-1/2}\cdot a_x(\overline{c})\,T^{-1/2}\,e^{T\psi(c^\ast_T)}\nonumber\\
        &=:f^{(6)}_{\beta,K}(T)\cdot a_x(\overline{c})\,T^{-1/2}\,e^{T\psi(c^\ast_T)}.
    \end{align}
    \textit{Error integral} ($D^{(7)}$).\\
    Lemma \ref{lem:errorint} gives us a constant $C^{(7)}>0$ and a $T^{(7)}\geq T^{(0)}$ such that
    \begin{align}\label{eq:E_T(H)boundexact}
        D^{(7)}=\limsup_{H\to\infty} \left|E_T(H)\right|
        &\leq C^{(7)}\, a_x(\overline{c})\,T^{-2}\,e^{T\psi(\overline{c})}\nonumber\\
        &= C^{(7)}\,T^{-3/2}\,e^{T(\psi(\overline{c})-\psi(c^\ast_T))}\,a_x(\overline{c})\,T^{-1/2}\,e^{T\psi(c^\ast_T)}\nonumber\\
        &\leq C^{(7)}\,T^{-3/2}\,e^{\frac{K^2}{4\theta^3}T^{2\beta-1}}\cdot a_x(\overline{c})\,T^{-1/2}\,e^{T\psi(c^\ast_T)}\nonumber\\
        &=:f^{(7)}_{\beta,K}(T)\cdot a_x(\overline{c})\,T^{-1/2}\,e^{T\psi(c^\ast_T)}.
    \end{align}
    for all $T\geq T^{(7)}$ and any $\abs x\abs\leq T^{\beta/2}\leq T^{1/2}$, where the last inequality again follows from \eqref{eq:psidiffbound}.\\\\
    After determining $f^{(i)}$ that satisfy \eqref{eq:Difirelation} for each $i=1,\ldots,7$ it remains to be seen that they are non-negative and vanish for $T\to\infty$.
    The non-negativity is trivial from the definitions.
    To check their behaviour for $T\to\infty$, note that
    $\beta<\frac12$ implies $T^{\beta-1/2},T^{2\beta-1}\to 0$ which in turn implies $e^{\frac{K}{4\theta^3}T^{2\beta-1}}\to 1$.
    Further, $\alpha\in(\frac13,\frac12)$ and $\alpha>\beta$ imply $T^{1-3\alpha}\to 0$, as well as $T^{\alpha-1/2}\to 0$ and $T^{\beta-\alpha}\to 0$, and
    $T^{1/2}e^{-\frac{\psi''(\overline{c})}{4}T^{1-2\alpha}}\to 0$, since $1-2\alpha>0$ and $\psi''(\overline{c})>0$. This shows that any term depending on $T$ in the definitions of $f^{(1)},\ldots, f^{(7)}$ 
    (i.e.\ \eqref{eq:I_T(H-r)boundexact}, \eqref{eq:I_T(r-Talpha)boundexact}, \eqref{eq:I_T(Talpha)axboundexact}), \eqref{eq:I_T(Talpha)psiboundexact}, \eqref{eq:D5boundexact}, \eqref{eq:asymptoticexact} and \eqref{eq:E_T(H)boundexact}) vanishes for $T\to\infty$ and we have proven Lemma \ref{lem:Difirelation}.
\end{proof}

\subsubsection{Explicit evaluation of the asymptotics}
Now, all that remains to be done is the explicit evaluation of the prefactor and asymptotic rate of the exponential in \eqref{eq:resultnoteval}. 
Plugging in the definition of $\overline{c}$ and $g(\overline{c})=\frac{1}{2\gamma\theta}$ into $a_x$ yields
\begin{align*}
    a_x(\overline{c})&=\frac{1}{\overline{c}}\sqrt{\frac{2g(\overline{c})}{1+g(\overline{c})}}\,\exp\left(-\frac{\gamma}{2}x^2(g(\overline{c})-1)\right)\\
    &=\frac{8\theta^2}{1-4\gamma^2\theta^2}\sqrt{\frac{\frac{1}{\gamma\theta}}{1+\frac{1}{2\gamma\theta}}}\,\exp\left(-\frac{\gamma}{2}x^2\left(\frac{1}{2\gamma\theta}-1\right)\right)\\
    &=\frac{\sqrt{2}\,8\theta^2}{\left(1+2\gamma\theta\right)^{3/2}\left(1-2\gamma\theta\right)}\,\exp\left(-\frac{x^2}{4\theta}\left(1-2\gamma\theta\right)\right).
\end{align*}
Continuing, by \eqref{eq:psiderivatives} we have
\begin{align*}
    \psi''(\overline{c})=\frac{1}{2\gamma^3 g(\overline{c})^3}=\frac{8\gamma^3\theta^3}{2\gamma^3}=4\theta^3
\end{align*}
and most importantly, by definition of $c^\ast_T$ and $g(c^\ast_T)=\frac{1}{2\gamma z}$,
\begin{align*}
    \psi(c_T^\ast)&=\frac{\gamma}{2}\left(1-g(c^\ast_T)\right)+z c^\ast_T
    =\frac{\gamma}{2}\left(1-\frac{1}{2\gamma z}\right)+z\frac{1-4\gamma^2 z^2}{8 z^2}\\
    &=\frac{2\gamma z-1}{4z}+\frac{1-4\gamma^2 z^2}{8z}
    =-\frac{4\gamma^2 z^2-4\gamma z+1}{8z}
    =-\frac{1}{8z}\left(2\gamma z-1\right)^2.
\end{align*}
This gives us the expression
\begin{align*}
    &\frac{1}{\sqrt{2\pi}}\frac{a_x(\overline{c})}{\sqrt{\psi''(\overline{c})}}\,T^{-1/2}\,e^{T\psi(c^\ast_T)}\\
    =~&\frac{1}{\sqrt{2\pi}}\frac{\sqrt{2}\,8\theta^2}{\left(1+2\gamma\theta\right)^{3/2}\left(1-2\gamma\theta\right)}\frac{1}{\sqrt{4\theta^3}}\,\exp\left(-\frac{x^2}{4\theta}\left(1-2\gamma\theta\right)\right)T^{-1/2}\,\exp\left({-T\frac{1}{8z}\left(1-2\gamma z\right)^2}\right)\\
    =~&\frac{1}{\sqrt{\pi}}\frac{4\sqrt{\theta}}{\left(1+2\gamma\theta\right)^{3/2}\left(1-2\gamma\theta\right)}\,T^{-1/2}\,\exp\left(-\frac{1}{8z}\left(1-2\gamma z\right)^2T-\frac{x^2}{4\theta}\left(1-2\gamma\theta\right)\right),
\end{align*}
which together with \eqref{eq:resultnoteval} completes the proof of Theorem \ref{thm:uniformasymptotics}.


\section{Proof of the Main Theorem}\label{sec:proofmainthm}

An Ornstein-Uhlenbeck process with parameter $\gamma$, started in $x$, can be defined as the strong solution of the stochastic differential equation
\begin{align}\label{def:SDE}
    dX_t = -\gamma X_t\,dt + dB_t,\quad t\geq 0,\quad\text{with}~ X_0=x,
\end{align}
where $B$ is a Brownian motion.
In the following, let $W$ be the Brownian motion for which $U$ is the strong solution of \eqref{def:SDE} and $\FF=\left(\FF_t\right)_{t\geq 0}$ its natural filtration.
Further, we want to consider $U$ as a random variable with values in $\CC([0,\infty))$ or its restriction to finite intervals.
We equip $\CC([0,\infty))$ with the topology of uniform convergence on compact subsets (cf.\ \cite{wcWhitt}) and the Borel $\sigma$-field it generates.\\

To examine the long term behavior of the Ornstein--Uhlenbeck process with restricted $L_2$-norm, we will first use its Markov property to separate the long term behaviour from that on a finite interval $[0,t_0]$ in the following lemma.

\begin{lem}\label{lem:markovProperty}
    Let $t_0>0$ and $y>0$. Then, for any Borel-measurable subset $D$ of $\CC([0,t_0])$, it holds for any $T\geq t_0$
    \begin{align}\label{eq:expectationofratiolem}
        \P_x(U\in D~\vert~Z_T\leq y)=\E_x\left[\1_{U\in D}~\frac{\P_{U_{t_0}}\left(Z_{T-t_0}\leq y-Z_{t_0}\right)}{\P_x(Z_T\leq y)}\right].
    \end{align}
\end{lem}
\begin{proof}
    Since
    \begin{align*}
        \int_{0}^{T} U_s^2\,ds = \int_{0}^{t_0} U_s^2\,ds+\int_{t_0}^{T} U_s^2\,ds 
        =\int_{0}^{t_0} U_s^2\,ds+ \int_{0}^{T-t_0} (U_{s+t_0})^2\,ds,
    \end{align*}
    we know that $Z_T\leq y$ if and only if
    \begin{align}
    \tilde{Z}_{T-t_0}\leq y-Z_{t_0},\qquad\text{with}~\tilde{Z}=\left(\int_{0}^{t}\tilde{U}_{s}^2\,ds\right)_{t\geq 0}
    \end{align}
    where $\tilde{U}=(U_{t+t_0})_{t\geq 0}$ is an Ornstein--Uhlenbeck process with parameter $\gamma$ started in $U_{t_0}$, by the Markov property of $U$.
    For any Borel-measurable $D\subset\CC\left([0,t_0]\right)$, since $U$ is adapted to $\FF$, we have $\{U\in D\}\in \FF_{t_0}$ and the observations above yield
    \begin{align}
        \P_x(U\in D,~Z_T\leq y)&=\E_x\left[\1_{U\in D}~\1_{Z_T\leq y}\right]
        =\E_x\left[~\E_x\left[\1_{U\in D}~\1_{Z_T\leq y}~\vert~\FF_{t_0}\right]\right]\nonumber\\
        &=\E_x\left[\1_{U\in D}~\E_x\left[\1_{\tilde{Z}_{T-t_0}\leq y-Z_{t_0}}~\vert~\FF_{t_0}\right]\right]\nonumber\\
        &=\E_x\left[\1_{U\in D}~\E_{U_{t_0}}\left[\1_{\tilde{Z}_{T-t_0}\leq y-Z_{t_0}(\omega)}\right]\right]\nonumber\\
        &=\E_x\left[\1_{U\in D}~\P_{U_{t_0}}\left(\tilde{Z}_{T-t_0}\leq y-Z_{t_0}(\omega)\right)\right],
    \end{align}
    where $Z_{t_0}(\omega)$ is supposed to indicate that the inner expectation and probability are taken only w.r.t. $\tilde{Z}$.
    From this it follows that
    \begin{align}\label{eq:expectationofratio}
        \P_x(U\in D~\vert~Z_T\leq y)=\E_x\left[\1_{U\in D}~\frac{\P_{U_{t_0}}\left(\tilde{Z}_{T-t_0}\leq y-Z_{t_0}(\omega)\right)}{\P_x(Z_T\leq y)}\right].
    \end{align}
    Since $\tilde{Z}$ and $Z$ are the same functional of an Ornstein-Uhlenbeck process with parameter $\gamma$, \eqref{eq:expectationofratio} is equivalent to the claim of the lemma.
\end{proof}

Since we have ``turned" the conditional probability into the ratio of two small ball probabilities, 
we will need the following uniform lemma on those ratios.
\begin{lem}\label{lem:ratiouniformasymptotics}
    Fix $x\in \R$ and $t_0>0$. Let $0<\theta<\frac{1}{2\gamma}$ and $\beta\in(0,\frac12)$. Then, for any $\delta\in(0,1)$ and $K>0$, there exists a $T_{\delta,\beta,K}>0$, such that for all $T\geq T_{\delta,\beta,K}$ and any $y_0,z\in\R$, $y>0$ with
    \[\abs z\abs\leq \frac{1}{2^{\beta/2}}T^{\beta/2}\quad\text{and}\quad \left|y-\theta T\right|\leq K\,T^\beta\quad\text{and}\quad 0\leq y_0\leq T^\beta\]
    it holds
    \begin{equation*}
        (1-\delta) h(z,y_0,t_o)\leq \frac{\P_z(Z_{T-t_0}\leq y-y_0)}{\P_x(Z_T\leq y)}
        \leq (1+\delta)h(z,y_0,t_0),
    \end{equation*}
    where
    \begin{equation*}
        h(z,y_0,t_0):=\exp\left((x^2-z^2+t_0)\frac{1-2\gamma\theta}{4\theta}- y_0\frac{1-4\gamma^2\theta^2}{8\theta^2}\right).
    \end{equation*}
\end{lem}
\begin{proof}
    We will only prove the lower bound, as it is the one we will need later on, the upper bound can be proven analogously.
    First, note that since $\beta<\frac12$ and $t_0$ is fixed, there exists a $T_0>0$ such that $(K+1)T^{\beta-1}\leq \frac{\theta}{2}$ and $T-t_0\geq \frac{T}{2}$ for all $T\geq T_0$. This implies
    \begin{align*}
        \abs y-y_0-\theta (T-t_0)\abs&\leq \abs y-\theta T\abs+ y_0+\theta t_0
        \leq 2^\beta(K+1)\left(\frac{T}{2}\right)^\beta +\theta t_0^{1-\beta} t_0^\beta\\
        &\leq 2^\beta (K+1)(T-t_0)^\beta + \theta t_0^{1-\beta}(T-t_0)^\beta
        =: M\,(T-t_0)^\beta,
    \end{align*}
    for all $T\geq T_0$. Further $\abs z\abs\leq \left(\frac{T}{2}\right)^{\beta/2}\leq(T-t_0)^{\beta/2}$.
    Hence, by Theorem \ref{thm:uniformasymptotics}, there exists a $T^{(1)}_{\delta,\beta,K}\geq T_0$ such that
    \begin{multline}\label{eq:sbprationumerator}
        \P_z(Z_{T-t_0}\leq y-y_0)\\
        \geq (1-\delta/3)\,C_\theta\,(T-t_0)^{-1/2}\,\exp\left(-\frac{1}{8\frac{y-y_0}{T-t_0}}\left(1-2\gamma\,\frac{y-y_0}{T-t_0}\right)^2(T-t_0)-z^2\,\frac{1-2\gamma\theta}{4\theta}\right),
    \end{multline}
    for all $T\geq T^{(1)}_{\delta,\beta,K}$ and any $z,y,y_0$ satisfying the assumptions of Lemma \ref{lem:ratiouniformasymptotics}, where $C_\theta$ is a constant not depending on $T$.
    Theorem \ref{thm:uniformasymptotics} also gives us a $T^{(2)}_{\delta,\beta,K}\geq T_0$ such that
    \begin{equation}\label{eq:sbpratiodenominator}
        \P_x(Z_T\leq y)
        \leq (1+\delta/3)\,C_\theta\,T^{-1/2}\,\exp\left(-\frac{1}{8\frac{y}{T}}\left(1-2\gamma \frac{y}{T}\right)^2T-x^2\,\frac{1-2\gamma\theta}{4\theta}\right),
    \end{equation}
    for all $T\geq T^{(2)}_{\delta,\beta,K}$ and any $y>0$ with $\abs y-\theta T \abs\leq KT^\beta$.
    Combining \eqref{eq:sbprationumerator} and \eqref{eq:sbpratiodenominator} therefore yields
    \begin{multline}
        \frac{\P_z(Z_{T-t_0}\leq y-y_0)}{\P_x(Z_T\leq y)}
        \geq \frac{1-\delta/3}{1+\delta/3}\,\left(\frac{T}{T-t_0}\right)^{1/2}\\
        \times\exp\left(-\frac{(T-t_0)^2}{8(y-y_0)}\left(1-2\gamma\,\frac{y-y_0}{T-t_0}\right)^2+\frac{T^2}{8y}\left(1-2\gamma\,\frac{y}{T}\right)^2+\left(x^2-z^2\right)\frac{1-2\gamma\theta}{4\theta}\right),
    \end{multline}
    for all $T\geq \max\{T^{(1)}_{\delta,\beta,K},T^{(2)}_{\delta,\beta,K}\}$ and any $z,y,y_0$ fulfilling the assumptions in the lemma.
    We now analyze the asymptotics of the right hand side of this inequality by first factoring out the right hand side of the claimed inequality and using $T/(T-t_0)\geq 1$:
    \begin{align}\label{eq:sbpratioboundexact}
        \frac{\P_z(Z_{T-t_0}\leq y-y_0)}{\P_x(Z_T\leq y)}
        \geq\frac{1-\delta/3}{1+\delta/3}\,\exp\left((x^2-z^2+t_0)\,\frac{1-2\gamma\theta}{4\theta}- y_0\,\frac{1-4\gamma^2\theta^2}{8\theta^2}\right)\times \exp(R),
    \end{align}
    where
    \begin{equation*}
        R:=-t_0\,\frac{1-2\gamma\theta}{4\theta}+y_0\,\frac{1-4\gamma^2\theta^2}{8\theta^2}-\frac{(T-t_0)^2}{8(y-y_0)}\left(1-2\gamma\,\frac{y-y_0}{T-t_0}\right)^2+\frac{T^2}{8y}\left(1-2\gamma\,\frac{y}{T}\right)^2.
    \end{equation*}
    If we can now show, that for some $T_{\delta,\beta,K}\geq \max\{T^{(1)}_{\delta,\beta,K},T^{(2)}_{\delta,\beta,K}\}$ 
    \begin{equation}\label{eq:exponentialasymptoticsbound}
        \exp(R)\geq 1-\delta/3,
    \end{equation}
    for all $T\geq T_{\delta,\beta,K}$, then the claim follows, as
    \begin{equation*}
        \frac{(1-\delta/3)^2}{1+\delta/3}=\frac{(1+\delta/3)^2-4\delta/3}{1+\delta/3}
        = 1+\delta/3-\frac{4\delta/3}{1+\delta/3}
        \geq 1+\delta/3-4\delta/3=1-\delta
    \end{equation*}
    and inserting into \eqref{eq:sbpratioboundexact} yields
    \begin{align}\label{eq:sbpratiobound}
        \frac{\P_z(Z_{T-t_0}\leq y-y_0)}{\P_x(Z_T\leq y)}
        &\geq\frac{1-\delta/3}{1+\delta/3}\,\exp\left((x^2-z^2+t_0)\,\frac{1-2\gamma\theta}{4\theta}- y_0\,\frac{1-4\gamma^2\theta^2}{8\theta^2}\right) (1-\delta/3)\nonumber\\
        &\geq (1-\delta)\,\exp\left((x^2-z^2+t_0)\,\frac{1-2\gamma\theta}{4\theta}- y_0\,\frac{1-4\gamma^2\theta^2}{8\theta^2}\right),
    \end{align}
    for all $T\geq T_{\delta,\beta,K}$ and any $z,y,y_0$ fulfilling the assumptions of the lemma.\\\\
    In order to show \eqref{eq:exponentialasymptoticsbound} we analyze $R$: Simple algebraic manipulations yield
    \begin{align}
        &-\frac{(T-t_0)^2}{8(y-y_0)}\left(1-2\gamma\,\frac{y-y_0}{T-t_0}\right)^2+\frac{T^2}{8y}\left(1-2\gamma\,\frac{y}{T}\right)^2\nonumber\\
        =~&-\frac{1}{8(y-y_0)}\left((T-2\gamma\,y)-(t_0-2\gamma\,y_0)\right)^2+\frac{1}{8y}\left(T-2\gamma\,y\right)^2\nonumber\\
        =~&-\frac{1}{8y(y-y_0)}\left(y\left((T-2\gamma\,y)-(t_0-2\gamma\,y_0)\right)^2-(y-y_0)\left(T-2\gamma\,y\right)^2\right)\nonumber\\
        =~&-\frac{1}{8y(y-y_0)}\left(-2yt_0T+4\gamma y^2t_0-4\gamma^2y^2y_0+yt_0^2-4\gamma yy_0t_0+4\gamma^2yy_0^2+y_0T^2\right)\nonumber\\
        =~&-\frac{1}{8y(y-y_0)}\left(-2yt_0T+4\gamma yt_0(y-y_0)-4\gamma^2yy_0(y-y_0)+yt_0^2+y_0T^2\right)\nonumber\\
        =~&-\frac{1}{8(y-y_0)}\left(-2t_0T+t_0^2\right)-\frac{\gamma}{2}t_0+\frac{\gamma^2}{2}y_0-\frac{y_0T^2}{8y(y-y_0)},
    \end{align}
    which implies
    \begin{align}\label{eq:expasymptoticsbound}
        \abs R\abs
        &=\left|-t_0\,\frac{1-2\gamma\theta}{4\theta}+y_0\,\frac{1-4\gamma^2\theta^2}{8\theta^2}-\frac{(T-t_0)^2}{8(y-y_0)}\left(1-2\gamma\,\frac{y-y_0}{T-t_0}\right)^2+\frac{T^2}{8y}\left(1-2\gamma\,\frac{y}{T}\right)^2\right|\nonumber\\
        &=\left|-t_0\,\frac{1-2\gamma\theta}{4\theta}+y_0\,\frac{1-4\gamma^2\theta^2}{8\theta^2}-\frac{1}{8(y-y_0)}\left(-2t_0T+t_0^2\right)-\frac{\gamma}{2}t_0+\frac{\gamma^2}{2}y_0-\frac{y_0T^2}{8y(y-y_0)}\right|\nonumber\\
        &=\left|\left(\frac{T}{4(y-y_0)}-\frac{1}{4\theta}\right)t_0+\left(\frac{1}{8\theta^2}-\frac{T^2}{8y(y-y_0)}\right)y_0-\frac{t_0^2}{8(y-y_0)}\right|\nonumber\\
        &\leq\left|\frac{T}{4(y-y_0)}-\frac{1}{4\theta}\right|t_0+\left|\frac{1}{8\theta^2}-\frac{T^2}{8y(y-y_0)}\right|y_0+\left|\frac{t_0^2}{8(y-y_0)}\right|.
    \end{align}
    We can bound each term in the last line separately, to do this we make the following observations using the definition of $T_0$: 
    Since $\abs y-y_0-\theta T\abs\leq (K+1)T^\beta$
    \begin{align}
        \frac{y-y_0}{T}\geq \theta-(K+1)T^{\beta-1}\geq \frac{\theta}{2}
    \end{align}
    for all $T\geq T_0$. Further, $y/T\geq (y-y_0)/T\geq \theta/2$ and $y/T\leq 3\theta/2$ for all $T\geq T_0$.
    Using these inequalities and the assumptions of the lemma we get the following three bounds:
    \begin{align}\label{eq:expasymptoticsbound1}
        \left|\frac{T}{4(y-y_0)}-\frac{1}{4\theta}\right|t_0
        &=\frac{t_0}{4\theta}\left|\frac{1}{y-y_0}\right|\left|y-y_0-\theta T\right|
        \leq \frac{t_0}{4\theta}\left|\frac{T}{y-y_0}\right|(K+1)T^{\beta-1}\nonumber\\
        &\leq \frac{t_0}{2\theta^2}(K+1)T^{\beta-1},
    \end{align}
    for the first summand in \eqref{eq:expasymptoticsbound} and for the second summand in \eqref{eq:expasymptoticsbound} we have
    \begin{align}\label{eq:expasymptoticsbound2}
        \left|\frac{1}{8\theta^2}-\frac{T^2}{8y(y-y_0)}\right|y_0
        &=\frac{y_0}{8\theta^2 \abs y(y-y_0)\abs}\left|y^2-yy_0-(\theta T)^2\right|\nonumber\\
        &\leq \frac{y_0}{8\theta^2 \abs y(y-y_0)\abs}\left(\left|y-\theta T\right|\left|y+\theta T\right|+\abs yy_0\abs\right)\nonumber\\
        &\leq \frac{y_0\,T^2}{8\theta^2 \abs y(y-y_0)\abs}\left(KT^{\beta-1}\left(\frac{y}{T}+\theta \right)+\frac{yy_0}{T^2}\right)\nonumber\\
        &\leq \frac{y_0}{2\theta^4}\left(KT^{\beta-1}\left(\frac{3\theta}{2}+\theta \right)+\frac{3\theta}{2} T^{\beta-1}\right)\nonumber\\
        &\leq\frac{5K+3}{4\theta^3}\,T^{2\beta-1},
    \end{align}
    for all $T\geq T_0$. Finally, we have, for all $T\geq T_0$,
    \begin{align}\label{eq:expasymptoticsbound3}
        \frac{t_0^2}{\abs 8(y-y_0)\abs}
        \leq \frac{t_0^2}{8}T^{-1}\frac{2}{\theta}
        =\frac{t_0^2}{4\theta}T^{-1}.
    \end{align}
    Inserting \eqref{eq:expasymptoticsbound1}, \eqref{eq:expasymptoticsbound2} and \eqref{eq:expasymptoticsbound3} into \eqref{eq:expasymptoticsbound} yields
    \begin{align}
        \abs R\abs
        &\leq\frac{t_0}{2\theta^2}(K+1)T^{\beta-1}+\frac{5K+3}{4\theta^3}\,T^{2\beta-1}+\frac{t_0^2}{4\theta}T^{-1}\nonumber\\
        &\leq\left(\frac{t_0}{2\theta^2}(K+1)+\frac{5K+3}{4\theta^3}+\frac{t_0^2}{4\theta}\right)T^{2\beta-1},
    \end{align}
    for $T\geq T_0$. Therefore, since $\beta<1/2$, there exists a $T_{\delta,\beta,K}\geq \max\{T^{(1)}_{\delta,\beta,K},T^{(2)}_{\delta,\beta,K}\}$ such that
    \begin{equation*}
        \abs R\abs\leq \frac{\delta}{6e}<\frac12,
    \end{equation*}
    for all $T\geq T_{\delta,\beta,K}$. Lemma \ref{lem:expTaylor} then yields
    \begin{equation*}
        \left|\exp\left(R\right)-1\right|\leq 2e\frac{\delta}{6e}=\frac{\delta}{3},
    \end{equation*}
    so that $\exp(R)\geq 1-\frac{\delta}{3}$
    for all $T\geq T_{\delta,\beta,K}$, which shows that the condition \eqref{eq:exponentialasymptoticsbound} is indeed satisfied.
    As argued in \eqref{eq:sbpratioboundexact} and \eqref{eq:sbpratiobound} this proves the claim.
\end{proof}

Before applying Lemma \ref{lem:ratiouniformasymptotics} to the expression in \eqref{eq:expectationofratiolem}, we first show another lemma, whose necessity will become apparent in the last step of the proof.

\begin{lem}\label{lem:girsanovdensitymartingale}
    Fix $x\in\R$ and define $M=(M_t)_{t\geq 0}$, with $h$ as in Lemma \ref{lem:ratiouniformasymptotics},
    \[M_t:=h\left(U_t,Z_t,t\right)=\exp\left(\left(x^2-U_t^2+t\right)\frac{1-2\gamma\theta}{4\theta}- Z_t\frac{1-4\gamma^2\theta^2}{8\theta^2}\right),\quad\forall t\geq 0.\]
    Then $M$ is a martingale and
    \[M_t=\exp\left(-\left(\frac{1}{2\theta}-\gamma\right)\int_0^tU_s\,d W_s-\frac12\left(\frac{1}{2\theta}-\gamma\right)^2Z_t\right),\quad\forall t\geq 0.\]
\end{lem}
\begin{proof}
    We show the claimed representation of $M$ first.
    For this purpose, let $Y^{(r)}=(Y^{(r)}_t)_{t\geq 0}$ with $Y^{(r)}_t:=-\left(\frac{1}{2\theta}-\gamma\right)\int_0^t \1_{s\geq r}U_s\,dW_s$ for all $r,t\geq 0$. 
    By Ito's Lemma, we have
    \begin{equation*}
        U_t^2-x^2=2\int_0^tU_s\,dU_s+\int_0^td\langle U\rangle_s
        =2\int_0^t-\gamma \,U_s^2\,ds+2\int_0^tU_s\,dW_s+\langle U\rangle_t
        =-2\gamma Z_t+2\int_0^tU_s\,dW_s+t.
    \end{equation*}
    Here we used that \eqref{def:SDE} for $U$ gives the formula $dU_s=-\gamma U_s\,ds+dW_s$ which also implies that the local martingale part of $U$'s semimartingale decomposition is $W$,
    and therefore $\langle U\rangle_t=\langle W\rangle_t=t$ for all $t\geq 0$.
    Substituting this into the definition of $M$ yields
    \begin{align}
        M_t&=\exp\left(\left(-2\int_0^tU_s\,dW_s+2\gamma Z_t\right)\frac{1-2\gamma\theta}{4\theta}- Z_t\frac{1-4\gamma^2\theta^2}{8\theta^2}\right)\nonumber\\
        &=\exp\left(-\frac{1-2\gamma\theta}{2\theta}\int_0^tU_s\,dW_s-Z_t\left(\frac{2\gamma^2\theta-\gamma}{2\theta}+\frac{1-4\gamma^2\theta^2}{8\theta^2}\right)\right)\nonumber\\
        &=\exp\left(-\left(\frac{1}{2\theta}-\gamma\right)\int_0^tU_s\,dW_s-\frac12 Z_t\left(\frac{4\gamma^2\theta^2-4\gamma\theta+1}{4\theta^2}\right)\right)\nonumber\\
        &=\exp\left(-\left(\frac{1}{2\theta}-\gamma\right)\int_0^tU_s\,dW_s-\frac12 Z_t\frac{(1-2\gamma\theta)^2}{4\theta^2}\right)\nonumber\\
        &=\exp\left(-\left(\frac{1}{2\theta}-\gamma\right)\int_0^tU_s\,dW_s-\frac12\left(\frac{1}{2\theta}-\gamma\right)^2Z_t\right),
    \end{align}
    for all $t\geq 0$.
    We notice that this representation of $M$ is the stochastic (Doléans-Dade-) exponential of $Y^{(0)}$:
    \begin{align}\label{eq:MisStochasticExp1}
        \exp\left(Y^{(0)}_t-\frac12\langle Y^{(0)}\rangle_t\right)
        &=\exp\left(-\left(\frac{1}{2\theta}-\gamma\right)\int_0^tU_s\,dW_s-\frac12\left(\frac{1}{2\theta}-\gamma\right)^2\int_0^tU_s^2\,ds\right)=M_t,
    \end{align}
    and for any $t\geq r\geq 0$ we analogously have
    \begin{align}\label{eq:MisStochasticExp2}
        \exp\left(Y^{(r)}_t-\frac12\langle Y^{(r)}\rangle_t\right)
        &=\exp\left(-\left(\frac{1}{2\theta}-\gamma\right)\int_r^t U_s\,dW_s-\frac12\left(\frac{1}{2\theta}-\gamma\right)^2\int_r^t U_s^2\,ds\right)
        =\frac{M_t}{M_r},
    \end{align}
    showing that $(M_t/M_r)_{t\geq r}$ is the stochastic exponential of $Y^{(r)}$ for $t\geq r$.
    With this observation, we can use Novikov's condition for martingality of stochastic exponentials.
    We note that $U_s\sim\NN\left(xe^{-\gamma s},\frac{1}{2\gamma}\left(
    1-e^{-2\gamma s}
    \right)\right)$ for all $s>0$.
    Together with Jensen's inequality and Fubini's theorem, for all $t>r\geq 0$ this yields
    \begin{align}\label{eq:NovikovsCondition}
        &\E\left[\exp\left(\frac12 \langle Y^{(r)}\rangle_t\right)\right]
        =\E\left[\exp\left(\frac12\left(\frac{1}{2\theta}-\gamma\right)^2\int_r^t U_s^2\,ds\right)\right]\nonumber\\
        =~&\E\left[\exp\left(\frac12\left(\frac{1}{2\theta}-\gamma\right)^2\int_0^1 (t-r)\,U_{s(t-r)+r}^2\,ds\right)\right]\nonumber\\
        \leq~& \E\left[\int_0^1\exp\left(\frac12\left(\frac{1}{2\theta}-\gamma\right)^2 (t-r)\,U_{s(t-r)+r}^2\right)\,ds\right]\nonumber\\
        =~&\frac{1}{t-r}\int_r^t\E\left[\exp\left(\frac12\left(\frac{1}{2\theta}-\gamma\right)^2 (t-r)\,U_{s}^2\right)\right]\,ds\nonumber\\
        =~&\frac{1}{t-r}\int_r^t \left(1-\left(\frac{1}{2\theta}-\gamma\right)^2\frac{1-e^{-2\gamma s}}{2\gamma}\,(t-r)\right)^{-\frac12}\exp\left(\frac{x^2\left(\frac{1}{2\theta}-\gamma\right)^2\,e^{-2\gamma s}\,(t-r)}{2-\left(\frac{1}{2\theta}-\gamma\right)^2\frac{1-e^{-2\gamma s}}{\gamma}\,(t-r)}\right)\,ds\nonumber\\
        \leq~& \frac{1}{t-r}\int_r^t \left(1-\left(\frac{1}{2\theta}-\gamma\right)^2\frac{t-r}{2\gamma}\right)^{-\frac12}\,\exp\left(\frac{x^2\left(\frac{1}{2\theta}-\gamma\right)^2\,(t-r)}{2-\left(\frac{1}{2\theta}-\gamma\right)^2\frac{t-r}{\gamma}}\right)\,ds\nonumber\\
        =~&\left(1-\left(\frac{1}{2\theta}-\gamma\right)^2\frac{t-r}{2\gamma}\right)^{-\frac12}\,\exp\left(\frac{x^2\left(\frac{1}{2\theta}-\gamma\right)^2\,(t-r)}{2-\left(\frac{1}{2\theta}-\gamma\right)^2\frac{t-r}{\gamma}}\right)<\infty,
    \end{align}
    for $t-r\leq\lambda$, where 
    \begin{equation*}
        \lambda:= \frac{\gamma}{\left(\frac{1}{2\theta}-\gamma\right)^2}<\frac{2\gamma}{\left(\frac{1}{2\theta}-\gamma\right)^2\left(1-e^{-2\gamma s}\right)}\quad \forall s>0.
    \end{equation*} 
    The second inequality here justifies the evaluation of the expectation in the fifth step in \eqref{eq:NovikovsCondition} using Lemma \ref{lem:generalchi^2mgf}.     This shows that for all $r\geq 0$, $Y^{(r)}$ satisfies Novikov's condition on $[r,r+\lambda]$ and therefore, together with \eqref{eq:MisStochasticExp2}, that $(M_t/M_r)_{t\in[r,r+\lambda]}$ is a martingale.
    As the stochastic exponential of a local martingale (which $Y^{(0)}$ is) is itself a non-negative local martingale, $M$ is a supermartingale.
    Therefore it now suffices to show that $M$ has constant expectation.
    Note that for any $t\in[0,\lambda]$, constant expectation follows directly from martingality of $(M_t)_{t\in[0,\lambda]}$.
    For $t\in[n\lambda,(n+1)\lambda]$, since $M$ is adapted to $\FF$ and $(M_t/M_{n\lambda})_{t\in[n\lambda,(n+1)\lambda]}$ is a martingale, we have, 
    \begin{align*}
        \E\left[M_t\right]
        =\E\left[\E\left[M_{n\lambda}\frac{M_t}{M_{n\lambda}}\bigg\vert \FF_{n\lambda}\right]\right]
        =\E\left[M_{n\lambda}\,\E\left[\frac{M_t}{M_{n\lambda}}\bigg\vert~\FF_{n\lambda}\right]\right]
        =\E\left[M_{n\lambda} \frac{M_{n\lambda}}{M_{n\lambda}}\right]
        =\E\left[M_{n\lambda}\right].
    \end{align*}
     This inductively shows, that $\E[M_t]=\E[M_0]=1$ for all $t\geq 0$, which as argued above shows the claim.
\end{proof}

We will now use Lemma \ref{lem:markovProperty} and Lemma \ref{lem:ratiouniformasymptotics} to examine the conditioned law of $U$ on a finite time interval.

\begin{lem}\label{lem:limfiniteinterval}
    Let $x\in\R$, $t_0>0$ and take $h$ as defined in Lemma \ref{lem:ratiouniformasymptotics}. Assume there is a $K>0$ and $\beta\in(0,1/2)$ such that for the sequence $y_T$ we have $|y_T-\theta T|\leq K T^\beta$. Then, for any Borel-measurable $D\subset\CC([0,t_0])$, we have
    \begin{equation}\label{eq:limfinalestimate}
        \lim_{T\to\infty}\P_x(U\in D~\vert~Z_T\leq y_T)=\E_x\left[\1_{U\in D}\,h(U_{t_0}, Z_{t_0}, t_0)\right].
    \end{equation}
\end{lem}
\begin{proof}
    Before applying Lemma \ref{lem:ratiouniformasymptotics} to the ratio in Lemma \ref{lem:markovProperty}, we need to make sure that the conditions of the lemma are satisfied uniformly.
    For this purpose, we define for a fixed $\beta\in(0,1/2)$ and $t_0>0$
    \begin{equation}\label{def:Gm}
        G_m:=\left\{f\in\CC\left([0,t_0]\right): \abs f_{t_0}\abs\leq \left(\frac{m}{2}\right)^\frac{\beta}{2},~~\int_0^{t_0}f_s^2\,ds\leq m^\beta\right\},~~m\in\N.
    \end{equation}
    With this definition, we have $\{U\in G_m\}\in\FF_{t_0}$ for any $m\in\N$ and for $U\in G_m$
    \begin{equation}
        \abs U_{t_0}\abs\leq\left(\frac{T}{2}\right)^\frac{\beta}{2},\quad \int_0^{t_0}U_s^2\,ds\leq T^\beta\quad\text{for any}~~T\geq m.
    \end{equation}
    In other words, we can apply the estimates of Lemma \ref{lem:ratiouniformasymptotics} if we restrict to $\{U\in G_m\}$:
    With $\delta\in(0,1)$, and $T_{\delta,\beta,K}>0$ as given by Lemma \ref{lem:ratiouniformasymptotics}, using Lemma \ref{lem:markovProperty} on $D\cap G_m$ we have
    \begin{align}
        &\P_x(U\in D~\vert~Z_T\leq y_T)\geq\P_x(U\in D\cap G_m~\vert~Z_T\leq y_T)\nonumber\\
        =~&\E_x\left[\1_{U\in D\cap G_m}~\frac{\P_{U_{t_0}}\left(Z_{T-t_0}\leq y_T-Z_{t_0}\right)}{\P_x(Z_T\leq y_T)}\right]
        \geq (1-\delta)\E_x\left[\1_{U\in D\cap G_m}\,h(U_{t_0}, Z_{t_0}, t_0)\right],
    \end{align}
    for all $T\geq\max\{T_{\delta,\beta,K}, m\}$ and any $m\in\N$.
    Since the right-hand side does not depend on $T$ anymore,
    \begin{equation}
        \liminf_{T\to\infty}\P_x(U\in D~\vert~Z_T\leq y_T)\geq (1-\delta)\E_x\left[\1_{U\in D\cap G_m}\,h(U_{t_0}, Z_{t_0}, t_0)\right],
    \end{equation}
    for any $m\in\N$. 
    Next, we note that any continuous path will eventually fulfill the conditions of $G_m$ for $m$ large enough, which implies $\bigcup_{m\in\N}G_m=\CC([0,t_0])$, and together with the non-negativity of $h$, this all implies
    \begin{equation*}
        \1_{U\in D\cap G_m}\,h(U_{t_0}, Z_{t_0}, t_0)\quad\overset{m\to\infty}{\longrightarrow}\quad \1_{U\in D}\,h(U_{t_0}, Z_{t_0}, t_0),
    \end{equation*}
    pointwise and monotonically increasing. Therefore, the monotone convergence theorem yields
    \begin{align}
        \liminf_{T\to\infty}\P_x(U\in D~\vert~Z_T\leq y_T)&\geq (1-\delta)\lim_{m\to\infty}\E_x\left[\1_{U\in D\cap G_m}\,h(U_{t_0}, Z_{t_0}, t_0)\right]\nonumber\\
        &=(1-\delta)\,\E_x\left[\1_{U\in D}\,h(U_{t_0}, Z_{t_0}, t_0)\right].
    \end{align}
    Now, letting $\delta\to 0$ yields
    \begin{equation}\label{eq:liminffinalestimate}
        \liminf_{T\to\infty}\P_x(U\in D~\vert~Z_T\leq y_T)\geq\E_x\left[\1_{U\in D}\,h(U_{t_0}, Z_{t_0}, t_0)\right].
    \end{equation}
    Since this holds for any Borel measurable set, it also holds for $D^c\in\BB(\CC([0,t_0]))$ and we get
    \begin{align}\label{eq:limsupfinalestimate}
        \limsup_{T\to\infty}\P_x(U\in D~\vert~Z_T\leq y_T)
        &=1-\liminf_{T\to\infty}\P_x(U\in D^c~\vert~Z_T\leq y_T)\nonumber\\
        &\leq 1-\E_x\left[\1_{U\in D^c}\,h(U_{t_0}, Z_{t_0}, t_0)\right]\nonumber\\
        &=1-\E_x\left[h(U_{t_0}, Z_{t_0}, t_0)\right]+\E_x\left[\1_{U\in D}\,h(U_{t_0}, Z_{t_0}, t_0)\right]\nonumber\\
        &=\E_x\left[\1_{U\in D}\,h(U_{t_0}, Z_{t_0}, t_0)\right].
    \end{align}
    In the last step we used that by Lemma \ref{lem:girsanovdensitymartingale} $\left(h(U_{t}, Z_{t}, t)\right)_{t\in[0,t_0]}$ is a martingale and therefore has constant expectation
    \begin{equation*}
    \E_x\left[h(U_{t_0}, Z_{t_0}, t_0)\right]=\E_x\left[h(U_{0}, Z_{0}, 0)\right]=\E_x\left[h(x, 0, 0)\right]=1.
\end{equation*}
    Combining \eqref{eq:liminffinalestimate} and \eqref{eq:limsupfinalestimate} then yields the claim.
\end{proof}

To finalize the proof of Theorem~\ref{thm:main}, take any $t_0>0$ and $x\in\R$.
Since, by Lemma \ref{lem:girsanovdensitymartingale},  $\left(h(U_{t}, Z_{t}, t)\right)_{t\in[0,t_0]}$ is a martingale, by Girsanov's theorem, $h(U_{t_0}, Z_{t_0}, t_0)$ is the density of some probability measure $\Q$ with respect to $\P$ on $\FF_{t_0}$.
Using this, we can reformulate the result of Lemma \ref{lem:limfiniteinterval}:
\begin{equation}\label{eq:weakConvergencePtoQ}
    \lim_{T\to\infty}\P_x(U\in D~\vert~Z_T\leq y_T)
    =\E_x\left[\1_{U\in D}\,h(U_{t_0}, Z_{t_0}, t_0)\right]
    =\E_x^\Q\left[\1_{U\in D}\right]
    =\Q_x(U\in D),
\end{equation}
for any $D\in\BB\left(\CC([0,t_0])\right)$.
Girsanov's theorem further yields that
\begin{equation*}
    \tilde{W}_t:=W_t-\langle W_\cdot, -\left(\frac{1}{2\theta}-\gamma\right)\int_0^\cdot U_s\,dW_s\rangle_t
    =W_t-\left(-\left(\frac{1}{2\theta}-\gamma\right)\int_0^t U_s\,ds\right),\quad t\in[0,t_0],
\end{equation*}
is a Brownian motion on $[0,t_0]$ with respect to $\Q$, as $h(U_{t}, Z_{t}, t)$ is the stochastic exponential of $-(1/2\theta-\gamma)\int_0^tU_s\,dW_s$ (cf. \eqref{eq:MisStochasticExp1}).
This allows us to examine which distribution $U$ has under $\Q$:
\begin{align}
    U_t-x&= -\gamma\int_0^tU_s\,ds + W_t
    =-\frac{1}{2\theta}\int_0^tU_s\,ds+ \left(\frac{1}{2\theta}-\gamma\right)\int_0^tU_s\,ds+W_t\nonumber\\
    &=-\frac{1}{2\theta}\int_0^tU_s\,ds+\tilde{W}_t, 
\end{align}
for all $t\in[0,t_0]$.
Since $\tilde{W}$ is a Brownian motion on $[0,t_0]$ under $\Q$, this shows that, with respect to $\Q$, $U$ is a solution to the stochastic differential equation
\begin{equation*}
    dX_t=-\frac{1}{2\theta}X_tdt+dB_t,\quad t\in[0,t_0],\quad\text{with}~~X_0=x,
\end{equation*}
where $B$ is some Brownian motion.
Since this is the stochastic differential equation defining an Ornstein--Uhlenbeck process with parameter $\frac{1}{2\theta}$ started in $x$,
\eqref{eq:weakConvergencePtoQ} shows the weak convergence of the conditioned measures to the distribution of said Ornstein--Uhlenbeck process on $\CC([0,t_0])$.
Since we have shown this for arbitrary $t_0>0$, Theorem 5 in \cite{wcWhitt} then implies that the same convergence also holds on $C([0,\infty))$, which was the claimed result.

\section*{Auxiliary bounds and identities}
Here, we collect some (in-)equalities that are regularly used.
\begin{lem}\label{lem:reg}
    Let $u=c+iv$ with $c>0$, $v\in \R$ and $g(u)=\left(1+\frac{2}{\gamma^2}u\right)^{1/2}$ (principal branch of the complex square root). We collect some useful bounds:
    \begin{enumerate}[(a)]
        \item $\Re g(u)\geq g(c)>1$,
        \item $\Re g(u)>\frac{1}{\sqrt{2}}\abs g(u)\abs$,
        \item $\Re g(u)>\frac{\sqrt{\abs v\abs}}{\gamma}$,
        \item $\Re g(u)\leq \frac{\sqrt{v}}{\gamma}\left(1+\frac{\gamma^2+2c}{r}\right)^{1/2}$ \quad for $v\geq r>0$,
        \item $\Re \frac{1}{g(u)}\geq0$,
        \item $\abs1-g(c+re^{it})\abs\leq\frac{1}{\gamma^2}(c+r)$\quad for $0<r\leq c$ and $t\in[0,2\pi]$. 
    \end{enumerate}
\end{lem}
\begin{proof}
    From the classical formula for the real part of the complex square root, we have
    \begin{equation}\label{eq:reg}
      \Re g(u)=\sqrt{\frac{\abs 1+\frac{2}{\gamma^2}u\abs+1+\frac{2}{\gamma^2}c}{2}}. 
    \end{equation}
    The first inequality follows from $\abs 1+\frac{2}{\gamma^2}(c+iv)\abs\geq\abs 1+\frac{2}{\gamma^2}c\abs$ and $c>0$,
    as this implies
    \[\Re g(u)\geq\sqrt{1+\frac{2}{\gamma^2}c}= g(c)\quad\text{and}\quad g(c)= \sqrt{1+\frac{2}{\gamma^2}c}>1.\]
    Since $c>0$, we also have
    \[\Re g(u)>\sqrt{\frac{\abs 1+\frac{2}{\gamma^2}u\abs}{2}}=\frac{\abs (1+\frac{2}{\gamma^2}u)^\frac12\abs}{\sqrt{2}}=\frac{\abs g(u)\abs}{\sqrt{2}}\]
    as well as
    \[\Re g(u)>\sqrt{\frac{\abs 1+\frac{2}{\gamma^2}(c+iv)\abs}{2}}>\sqrt{\frac{\abs \frac{2}{\gamma
    ^2}v\abs}{2}}=\frac{\sqrt{\abs v\abs}}{\gamma}.\]
    The other direction of this bound is a bit more involved. By the triangle inequality, \\$\abs 1+\frac{2}{\gamma^2}(c+iv)\abs\leq 1+\frac{2}{\gamma^2}(c+v)$ for $v>0$. This implies
    \begin{align*}
        \Re g(u)&\leq \left(\frac{1+\frac{2}{\gamma^2}(c+v)+1+\frac{2}{\gamma^2}c}{2}\right)^{1/2}
        =\left(1+\frac{2}{\gamma^2}c+\frac{1}{\gamma^2}v\right)^{1/2}\\
        &=\frac{\sqrt{v}}{\gamma}\left(1+\frac{2c+\gamma^2}{v}\right)^{1/2}
        \leq\frac{\sqrt{v}}{\gamma}\left(1+\frac{2c+\gamma^2}{r}\right)^{1/2}
    \end{align*}
    for $v\geq r>0$.
    The classical formula for the real part of a complex fraction yields
    \[\Re \frac{1}{g(u)}=\frac{\Re g(u)}{\abs g\abs^2}>\frac{1}{\abs g\abs^2}\geq0.\]
    
    The last claim follows from
    \begin{align*}
        \left|1-g(c+re^{it})\right|=\left|\frac{1-g(c+re^{it})^2}{1+g(c+re^{it})}\right|
        =\frac{\frac{2}{\gamma^2}\left|c+re^{it}\right|}{\left|1+g(c+re^{it})\right|}
        \leq \frac{\frac{2}{\gamma^2}(c+r)}{1+\Re g(c+re^{it})}\leq\frac{1}{\gamma^2}(c+r)
    \end{align*}
    as $\Re (c+re^{it})\geq 0$ implies $\Re g(c+re^{it})\geq 1$. This follows from the same idea as the proof of part (a)
    with $c'=c+r\cos(t)\geq 0$ (as $0<r\leq c$) and $v'=r\sin(t)$:
    \[\Re g(c+re^{it})=\Re g(c'+iv')\geq \sqrt{1+\frac{2c'}{\gamma^2}}\geq 1. \qedhere \]
\end{proof}
\begin{lem}\label{lem:regimcomponent}
    With $g$ defined as in Lemma \ref{lem:reg}, for any $c>0$ and $w>0$, there exists a $C_{c,w}>0$ only depending on $c,w$ such that for all $v\geq w>0$
    \[\Re g(c+iv)-\Re g(c+iw)\geq C_w(\sqrt{v}-\sqrt{w}).\]
\end{lem}
\begin{proof}
    By Lemma \ref{lem:reg}(d), 
    \begin{align*}
        \left(\Re g(c+iv)+\Re g(c+iw)\right)
        &\leq \left(\frac{\sqrt{v}}{\gamma}\left(1+\frac{\gamma^2+2c}{v}\right)^{1/2}+\frac{\gamma\,\Re g(c+iw)}{\sqrt{v}}\frac{\sqrt{v}}{\gamma}\right)\\
        &\leq \frac{1}{\gamma}\sqrt{v}\left(\left(1+\frac{\gamma^2+2c}{w}\right)^{1/2}+\frac{\gamma\,\Re g(c+iw)}{\sqrt{w}}\right)=:\frac{1}{\gamma}\sqrt{v}\,C.
    \end{align*}
    Obviously, $C$ is positive and only depends on $c$ and $w$.
    Expanding by $\Re g(c+iv)+\Re g\left(c+iw\right)$ and using \eqref{eq:reg} yields
    \begin{align}\label{eq:diffreg}
        \Re g(c+iv)-\Re g\left(c+iw\right)
        &=\frac{\left(\Re g(c+iv)\right)^2-\left(\Re g\left(c+iw\right)\right)^2}{\Re g(c+iv)+\Re g(c+iw)}\nonumber\\
        &=\frac{\abs1+\frac{2}{\gamma^2}(c+iv)\abs+1+\frac{2}{\gamma^2}c-\left(\abs1+\frac{2}{\gamma^2}(c+iw)\abs+1+\frac{2}{\gamma^2}c\right)}{2\left(\Re g(c+iv)+\Re g(c+iw)\right)}\nonumber\\
        &\geq \frac{\abs1+\frac{2}{\gamma^2}(c+iv)\abs-\abs1+\frac{2}{\gamma^2}(c+iw)\abs}{\frac{2}{\gamma}\sqrt{v}\,C}.
    \end{align}
    We want to repeat this idea of expanding to simplify the numerator, therefore we need 
    \begin{align*}
        \left|1+\frac{2}{\gamma^2}(c+iv)\right|+\left|1+\frac{2}{\gamma^2}(c+iw)\right|
        &\leq 1+\frac{2}{\gamma^2}(c+v)+1+\frac{2}{\gamma^2}(c+w)= \frac{2v}{\gamma^2}\left(\frac{\gamma^2}{v}+\frac{2c}{v}+1+\frac{w}{v}\right)\\
        &\leq \frac{2v}{\gamma^2}\left(2+\frac{2c+\gamma^2}{w}\right)=:\frac{2v}{\gamma^2}C'.
    \end{align*}
    Again, $C'$ is clearly positive and only depending on $c$ and $w$.
    Combined with \eqref{eq:diffreg}, we have
    \begin{align*}
        \Re g(c+iv)-\Re g\left(c+iw\right)&
        \geq \frac{\abs1+\frac{2}{\gamma^2}(c+iv)\abs-\abs1+\frac{2}{\gamma^2}(c+iw)\abs}{\frac{2}{\gamma}\sqrt{v}\,C}\\
        &=\frac{\left(1+\frac{2}{\gamma^2}c\right)^2+\frac{4}{\gamma^4}v^2-\left(\left(1+\frac{2}{\gamma^2}c\right)^2+\frac{4}{\gamma^4}w^2\right)}{\frac{2}{\gamma}\sqrt{v}\,C\left(\left|1+\frac{2}{\gamma^2}(c+iv)\right|+\left|1+\frac{2}{\gamma^2}(c+iw)\right|\right)}\\
        &\geq \frac{\frac{4}{\gamma^4}v^2-\frac{4}{\gamma^4}w^2}{\frac{2}{\gamma}\sqrt{v}\,C\, \frac{2}{\gamma^2}v\,C'}
        =\frac{1}{\gamma\,C\,C'}\left(\sqrt{v}\,-\frac{w^2}{v^{3/2}}\right)\geq\frac{1}{\gamma\,C\,C'}\left(\sqrt{v}\,-\sqrt{w}\right).
    \end{align*}
    This is the claimed result with $C_{c,w}:=\frac{1}{\gamma\,C\,C'}$, which is positive and depends only on $c$ and $w$.
\end{proof}
Finally, we collect some complex Taylor series expansions:
\begin{lem}\label{lem:complextaylorestimates}
    Fix $r>0$ and $z_0\in\C$. Let $f$ be holomorphic on $B_{r'}(z_0)$ for some $r'>r$. Then
    \[f(z)=\sum_{k=0}^\infty \frac{f^{(k)}(z_0)}{k\,!}(z-z_0)^k\]
    for all $z\in B_r(z_0)$ and
    \[\left|f(z)-\sum_{k=0}^n \frac{f^{(k)}(z_0)}{k\,!}(z-z_0)^k\right|\leq \max_{\abs z'-z_0\abs=r}\abs f(z')\abs\frac{\abs z-z_0\abs^{n+1}}{r^{n+1}}\frac{1}{1-\frac{\abs z-z_0\abs}{r}}.\]
\end{lem}
\begin{proof}
    it is well known that a holomorphic function can be written as its Taylor series, see Section 2.4 in \cite{Bourchtein} for example.
    Using the maximum modulus theorem and Cauchy's estimate (10.24 and 10.26 in \cite{Rudin}), it holds
    \begin{align*}
        &\left|f(z)-\sum_{k=0}^n \frac{f^{(k)}(z_0)}{k\,!}(z-z_0)^k\right|
        \leq\sum_{k=n+1}^\infty \left|f^{(k)}(z_0)\right|\frac{\abs z-z_0\abs^k}{k\,!}\\
        &\leq\sum_{k=n+1}^\infty \, \max_{z'\in B_r(z_0)}\,\abs f(z')\abs \frac{k!}{r^k}\frac{\abs z-z_0\abs^k}{k\,!}
        =\sum_{k=n+1}^\infty \, \max_{\abs z'-z_0\abs=r}\,\abs f(z')\abs\frac{\abs z-z_0\abs^k}{r^k}\\
        &= \max_{\abs z'-z_0\abs=r}\,\abs f(z')\abs\, \frac{\abs z-z_0\abs^{n+1}}{r^{n+1}}\sum_{k=n+1}^\infty \frac{\abs z-z_0\abs^{k-(n+1)}}{r^{k-(n+1)}}\\
        &= \max_{\abs z'-z_0\abs=r}\,\abs f(z')\abs\, \frac{\abs z-z_0\abs^{n+1}}{r^{n+1}}\sum_{k=0}^\infty \frac{\abs z-z_0\abs^{k}}{r^{k}}\\
        &= \max_{\abs z'-z_0\abs=r}\,\abs f(z')\abs \frac{\abs z-z_0\abs^{n+1}}{r^{n+1}}\frac{1}{1-\frac{\abs z-z_0\abs}{r}}. \qedhere
    \end{align*}
\end{proof}
Applying Lemma \ref{lem:complextaylorestimates} to the exponential function gives the following fact.

\begin{lem}\label{lem:expTaylor}
    On the complex unit disk $B_1(0)$, the exponential function has the first order Taylor expansion
    \[e^z=1+R_0(z),\quad \abs R_0(z)\abs\leq e\frac{\abs z\abs}{1-\abs z\abs},\]
    and therefore, for any $z$ with $\abs z\abs\leq r<1$
    \[\abs e^z-1\abs=\abs R_0(z)\abs\leq \frac{e}{1-r}\abs z\abs.\]
\end{lem}
Similarly, we can treat the function $f(z):=(1+z)^{-1/2}$.
\begin{lem}\label{lem:invsqrtTaylor}
    Let $f(z):=(1+z)^{-1/2}$. For any $\abs z\abs\leq r<1$
    \[f(z)=1+R_0(z),\quad \abs R_0(z)\abs\leq\frac{1}{\sqrt{1-r}}\frac{\abs z\abs}{r}\frac{1}{1-\frac{\abs z\abs}{r}},\]
    and therefore, for any $\abs z\abs\leq r'<r$
    \[\abs f(z)-1\abs\leq\frac{1}{\sqrt{1-r}}\frac{1}{r-r'}\abs z\abs.\]
\end{lem}
\begin{proof}
    Since $f(0)=1$, we have $f(z)=1+R_0(z)$. As
    \[\abs (1+re^{it})\abs^\frac12\geq(1+r\cos(t))^\frac12\geq(1-r)^\frac12\]
    we have \[\max_{\abs z\abs=r}\abs f(z)\abs=\max_{t\in[0,2\pi]}\abs f(re^{it})\abs=\max_{t\in[0,2\pi]}\frac{1}{\abs 1+re^{it}\abs^{-1/2}}\leq\frac{1}{\sqrt{1-r}},\]
    with which we get the first claim from Lemma \ref{lem:complextaylorestimates}.
    The second claim then follows immediately from the first and $\frac{1}{1-\frac{\abs z\abs}{r}}\leq\frac{1}{1-\frac{r'}{r}}$ for any $\abs z\abs\leq r'$.
\end{proof}
Since the real part of the complex root essentially behaves like a square root in the imaginary input (cf. Lemma \ref{lem:reg}(c),(d)), we need to evaluate the following integral a few times in Section~\ref{sec:unifsbp}.
\begin{lem}\label{lem:int}
    Let $a,m>0$. Then
    \[\int_0^a e^{-m\sqrt{x}}\,dx=\frac{2}{m^2}\left(1-e^{-m\sqrt{a}}(m\sqrt{a}+1)\right).\]
\end{lem}
\begin{proof}
    Substituting $y=e^{-m\sqrt{x}}$ yields $dy=\frac{-m}{2\sqrt{x}}e^{-m\sqrt{x}}\,dx$ which is equivalent to 
    $\frac{2}{m^2}\ln(y)\,dy=e^{-m\sqrt{x}}\,dx$. The integral therefore transforms to
    \[\int_0^a e^{-m\sqrt{x}}\,dx=\int_1^{e^{-m\sqrt{a}}}\frac{2}{m^2}\ln(y)\,dy=\frac{2}{m^2}\bigg(y\ln(y)-y\big\vert_1^{e^{-m\sqrt{a}}}\bigg)=\frac{2}{m^2}\left(e^{-m\sqrt{a}}(-m\sqrt{a}-1)+1\right).\]
\end{proof}

Finally, we need the moment generating function of the square of a general normal distribution. 

\begin{lem}\label{lem:generalchi^2mgf}
    Let $X\sim\NN(\mu,\sigma^2)$. The moment generating function of $X^2$ is given by
    \begin{equation*}
        \E\left[\exp\left(tX^2\right)\right]
        =\left(1-2\sigma^2 t\right)^{-\frac12}\,\exp\left(\frac{\mu^2 t}{1-2\sigma^2 t}\right)\quad\forall~t<\frac{1}{2\sigma^2}.
    \end{equation*}
\end{lem}
\begin{proof}
    Let $d:=1-2\sigma^2 t$. As 
    \begin{align*}
        tx^2-\frac{(x-\mu)^2}{2\sigma^2}
        &=-\frac{1}{2\sigma^2}\left(dx^2-2x\mu+\mu^2\right)
        =-\frac{d}{2\sigma^2}\left(x^2-2x\frac{\mu}{d}+\left(\frac{\mu}{d}\right)^2-\left(\frac{\mu}{d}\right)^2+\frac{\mu^2}{d}\right)\\
        &=-\frac{d}{2\sigma^2}\left(x-\frac{\mu}{d}\right)^2+\frac{\mu^2}{2\sigma^2}\left(\frac{1}{d}-1\right),
    \end{align*}
    we have
    \begin{align*}
        \E\left[\exp\left(tX^2\right)\right]
        &=\frac{1}{\sqrt{2\pi\sigma^2}}\int_{-\infty}^\infty e^{tx^2}\,e^{-\frac{(x-\mu)^2}{2\sigma^2}}\,dx
        =e^{\frac{\mu^2}{2\sigma^2}\left(\frac{1}{d}-1\right)}\,\frac{1}{\sqrt{2\pi\sigma^2}}\int_{-\infty}^\infty e^{-\frac{d}{2\sigma^2}\left(x-\frac{\mu}{d}\right)^2}\,dx\\
        &=\frac{1}{\sqrt{d}}e^{\frac{\mu^2}{2\sigma^2}\left(\frac{1-d}{d}\right)}\,\frac{1}{\sqrt{2\pi\frac{\sigma^2}{d}}}\int_{-\infty}^\infty e^{-\frac{d}{2\sigma^2}\left(x-\frac{\mu}{d}\right)^2}\,dx
        =\frac{1}{\sqrt{1-2\sigma^2 t}}e^{\frac{\mu^2}{2\sigma^2}\left(\frac{2\sigma^2t}{1-2\sigma^2t}\right)}\\
        &=\left(1-2\sigma^2 t\right)^{-\frac12}\,\exp\left(\frac{\mu^2 t}{1-2\sigma^2 t}\right),
    \end{align*}
    for all $\sigma^2 t<\frac12$.
\end{proof}

\newpage
\bibliographystyle{plain}
\bibliography{references}

@unpublished{papertobiasschmidt,
  author = "Schmidt, Tobias",
  title  = "Conditioned {B}rownian {M}otion and {L}ocal {E}quivalence of {P}ath {E}nsembles",
  note="Preprint,	arXiv:2608.19744"
}

@article{doob57,
  author    = {Doob, Joseph L.},
  title     = {Conditional {B}rownian motion and the boundary limits of harmonic functions},
  journal   = {Bulletin de la Soci{\'e}t{\'e} Math{\'e}matique de France},
  volume    = {85},
  pages     = {431--458},
  year      = {1957},
  publisher = {Soci{\'e}t{\'e} math{\'e}matique de France},
  url       = {https://www.numdam.org/item/BSMF_1957__85__431_0/}
}

@article{Fatalov2009,
  author  = {Fatalov, Vladimir R.},
  title   = {Occupation Time and Exact Asymptotics of Distributions of
             {$L^p$}-Functionals of the {Ornstein--Uhlenbeck} Processes,
             {$p>0$}},
  journal = {Theory of Probability and Its Applications},
  volume  = {53},
  number  = {1},
  pages   = {13--36},
  year    = {2009},
  doi     = {10.1137/S0040585X97983407}
}

@article{BercuRouault2002,
  author  = {Bercu, Bernard and Rouault, Alain},
  title   = {Sharp Large Deviations for the {Ornstein--Uhlenbeck} Process},
  journal = {Theory of Probability and Its Applications},
  volume  = {46},
  number  = {1},
  pages   = {1--19},
  year    = {2002},
  doi     = {10.1137/S0040585X97978737}
}

@article{ChetriteTouchette2015,
  author  = {Chetrite, Rapha{\"e}l and Touchette, Hugo},
  title   = {Nonequilibrium {M}arkov Processes Conditioned on Large Deviations},
  journal = {Annales Henri Poincar{\'e}},
  volume  = {16},
  number  = {9},
  pages   = {2005--2057},
  year    = {2015},
  doi     = {10.1007/s00023-014-0375-8}
}

@article{BrycDembo1997,
  author  = {Bryc, W{\l}odzimierz and Dembo, Amir},
  title   = {Large Deviations for Quadratic Functionals of {G}aussian Processes},
  journal = {Journal of Theoretical Probability},
  volume  = {10},
  number  = {2},
  pages   = {307--332},
  year    = {1997},
  doi     = {10.1023/A:1022656331883}
}

@article{BercuGamboaLavielle2000,
  author  = {Bercu, Bernard and Gamboa, Fabrice and Lavielle, Marc},
  title   = {Sharp {L}arge {D}eviations for {G}aussian {Q}uadratic {F}orms with {A}pplications},
  journal = {ESAIM: Probability and Statistics},
  volume  = {4},
  pages   = {1--24},
  year    = {2000}
}

@article{Dankel1991,
  author  = {Dankel Jr., Thad},
  title   = {On the Distribution of the Integrated Square of the
             {Ornstein--Uhlenbeck} Process},
  journal = {SIAM Journal on Applied Mathematics},
  volume  = {51},
  number  = {2},
  pages   = {568--574},
  year    = {1991},
  doi     = {10.1137/0151029}
}

@article{duBuissonTouchette2023,
  author  = {{du Buisson}, Johan and Touchette, Hugo},
  title   = {Dynamical Large Deviations of Linear Diffusions},
  journal = {Physical Review E},
  volume  = {107},
  number  = {5},
  pages   = {054111},
  year    = {2023},
  doi     = {10.1103/PhysRevE.107.054111}
}

@article{wcWhitt,
author = {Whitt, Ward},
title = {Weak Convergence of Probability Measures on the Function Space $C\lbrack 0, \infty)$},
volume = {41},
journal = {The Annals of Mathematical Statistics},
number = {3},
publisher = {Institute of Mathematical Statistics},
pages = {939 -- 944},
year = {1970},
doi = {10.1214/aoms/1177696970},
URL = {https://doi.org/10.1214/aoms/1177696970}
}

@book{BorodinSalminen,
author = {Borodin, Andrei N. and Salminen, Paavo},
year = {2002},
title = {Handbook of Brownian Motion — Facts and Formulae},
publisher = {Birkhäuser Basel},
isbn = {978-3-7643-6705-3},
doi = {10.1007/978-3-0348-8163-0}
}

@article{KnightBM,
 author = {Knight, Frank B.},
 journal = {Transactions of the American Mathematical Society},
 pages = {173--185},
 publisher = {American Mathematical Society},
 title = {Brownian Local Times and Taboo Processes},
 volume = {143},
 year = {1969}
}

@article{ALSbmtoou,
  author  = {Aurzada, Frank and Lifshits, Mikhail and Schickentanz, Dominic T.},
  title   = {Brownian motion conditioned to have restricted {$L_2$}-norm},
  journal = {Annales de l'Institut Henri Poincar{\'e}, Probabilit{\'e}s et Statistiques},
  volume  = {62},
  number  = {2},
  pages   = {941--955},
  year    = {2026},
  doi     = {10.1214/24-AIHP1532}
}

@book{Bourchtein,
author = {Bourchtein, Andrei and Bourchtein, Ludmila},
year = {2021},
title = {Complex Analysis},
publisher = {Springer Singapore},
isbn = {978-981-15-9218-8},
doi = {10.1007/978-981-15-9219-5}
}

@book{Rudin,
author = {Rudin, Walter},
year = {1987},
title = {Real and Complex Analysis},
publisher = {McGraw-Hill},
isbn = {0-07-100276-6},
}

@article{NazarovPetrova2023,
  author  = {Nazarov, Alexander I. and Petrova, Yulia},
  title   = {{$L_2$}-Small Ball Asymptotics for {G}aussian Random Functions:
             A Survey},
  journal = {Probability Surveys},
  volume  = {20},
  pages   = {608--663},
  year    = {2023},
  doi     = {10.1214/23-PS20}
}

@article{Li2001,
  author  = {Li, Wenbo V.},
  title   = {Small {B}all {P}robabilities for {G}aussian {M}arkov {P}rocesses
             under the {$L^p$}-{N}orm},
  journal = {Stochastic Processes and their Applications},
  volume  = {92},
  number  = {1},
  pages   = {87--102},
  year    = {2001},
  doi     = {10.1016/S0304-4149(00)00072-7}
}

\end{document}